\documentclass[twoside,11pt]{article}

\usepackage{amssymb,amsthm}
\usepackage[colorlinks,urlcolor=green]{hyperref}
\usepackage{graphicx}
\usepackage{multirow}
\usepackage{float}
\usepackage{caption}
\usepackage{epstopdf}
\usepackage{enumerate}
\usepackage{array}
\usepackage{booktabs}
\usepackage{color}
\usepackage{mathrsfs}   
\usepackage{amsmath} 
\usepackage{amsfonts} 
\usepackage{amssymb} 
\usepackage{hyperref}
\usepackage{eucal}
\usepackage{epstopdf}

\newtheorem{theorem}{Theorem}[section]

\newtheorem{lemma}[theorem]{Lemma}

\theoremstyle{definition}
\theoremstyle{plain}

\theoremstyle{plain}

\numberwithin{equation}{section}
\begin{document}
	\pagestyle{myheadings}
	\begin{titlepage}
		\title{\bf  Hamiltonian paths extending a set of matchings in hypercubes}
		\author{Abid Ali$^1$, Lina Ba$^2$ {\small and} Weihua Yang$^{1,*}$  \\
			{\it \small $^1$College of Mathematics, Taiyuan University of Technology, Taiyuan, Shanxi, PR China, 030024}\\
			{\it\small $^2$School of Finance and Mathematics Huainan Normal University,  Huainan, Anhui, PR China, 232038}\\
			{\it \small Emails:\ }
			{\it \small lian0003@link.tyut.edu.cn, baln19@lzu.edu.cn, yangweihua@tyut.edu.cn}\\
		} 
	\end{titlepage}
	\date{}
	\maketitle \vskip 0.1cm \centerline
	\hrulefill
\begin{abstract}	
The hypercube \( Q_n \) contains a Hamiltonian path joining \( x \) and \( y \) (where $x$ and $y$ from the opposite partite set) containing \( P \) if and only if the induced subgraph of \( P \) is a linear forest, where none of these paths have \( x \) or \( y \) as internal vertices nor both as endpoints. Dvořák and Gregor answered a problem posed by Caha and Koubek and proved that for every \( n \geq 5 \), there exist vertices \( x \) and \( y \) with a set of \( 2n - 4 \) edges in \( Q_n \) that extend to the Hamiltonian path joining \( x \) and \( y \). This paper examines the Hamiltonian properties of hypercubes with a matching set. Let consider the hypercube \( Q_n \), for \( n \geq 5 \) and a set of matching \( M \) such that \( |M| \leq 3n - 13 \). We prove a Hamiltonian path exists joining two vertices $x$ and $y$ in \( Q_n \) from opposite partite sets containing $M$. 
\end{abstract}
	\noindent \emph{\bf Keywords}:~{Matching, Hamiltonian path,  Hypercube, Hamiltonian laceable}.
	
	\maketitle
	\hrulefill
	
\section{Introduction}
The $n$-dimensional hypercube $Q_n$ is known to be Hamiltonian for all $n \geq 2$. This publication of the fact dates back to 1872 \cite{1}, and since the study of properties of the Hamiltonian cycle has received considerable attention \cite{ DT05, Dt09, SC97, Kre96, Fin07,  Fink09, Gre09, 11,  TJ18}. A classical result by Havel \cite{H84} shows that for every two vertices $x$ and $y$ in $Q_n$ with $n \geq 1$, if  $x$ and $y$ from the distinct partite set i.e., $d(x,y)$ is odd, then a Hamiltonian path exists between them. Note that the above condition is trivially necessary, as hypercubes are bipartite graphs with an even number of vertices.

In a graph \( G \), a path (respectively, cycle)  is classified as a Hamiltonian path (respectively, cycle) if every vertex of \( G \) appears in path (respectively, cycle) exactly one time. Particularly, a Hamiltonian path visits every vertex only once without forming a cycle, whereas a Hamiltonian cycle returns to its starting vertex. A graph is Hamiltonian if it has a Hamiltonian cycle. Furthermore, if a Hamiltonian path exists between every two distinct vertices in the graph, the graph is considered Hamiltonian-connected. A bipartite graph, \( G = (V_0 \cup V_1, E) \) is one in which the vertex set \( V(G) \) can be divided into two disjoint subsets, \( V_0 \) and \( V_1 \), such that each edge in \( G \) connects a vertex from \( V_0 \) to a vertex in \( V_1 \). A fundamental question in studying the Hamiltonian properties of a graph is determining whether it is Hamiltonian-connected or Hamiltonian. Although in Hamiltonian bipartite graphs  \( |V_0| = |V_1| \), as a Hamiltonian cycle must alternate between vertices in \( V_0 \) and \( V_1 \). Consequently, not all Hamiltonian bipartite graphs are Hamiltonian-connected. For those Hamiltonian bipartite graphs, the concept of Hamiltonian laceability was presented by Simmons \cite{new2}. A bipartite graph \( G = (V_0 \cup V_1, E) \), if connecting any two vertices \( x \) and \( y \) such that \( y \in V_1 \) and \( x \in V_0 \), is Hamiltonian laceable.
	
A problem introduced by Caha and Koubek:  for every two vertices $x$ and $y$ of distinct partite sets, does there exist a Hamiltonian path joining $x$ and $y$ that passes through a given subset of edges in $Q_n$? They observed in \cite{DT05}, that any such subset $P$ of edges from a Hamiltonian path joining $x$ and $y$ must induce a subgraph $P$ is a linear forest, none of these paths connect $x$ and $y$, neither $x$ nor $y$ can be incident with more than one edge of $P$. Moreover, when $|P| \leq n-2$ for $n \geq 2$, assuming $x$ and $y$ are two vertices of opposite partite sets, the above conditions say "\textit{natural necessary condition}" are not only necessary but also sufficient for ensuring the existence of a Hamiltonian path containing all edges of $P$. However, they also provided a counterexample for $ n\geq 3$, showing a set of $ 2n-3$ edges satisfying the above condition that cannot extend to any Hamiltonian path between $x$ and $y$, are from distinct partite sets. Indeed, let $u \neq y$ be a neighbor of $x$, and construct a set $P$ of edges connected to neighbors of $u$ such that $x$ is incident with one edge of $P$, all but one of the remaining neighbors are incident with two edges from $P$, and $u$ is not incident with any edge in $P$, satisfies the above conditions. Since each vertex in $Q_n$ has degree $n$, this implies $|P| = 2n - 3$. Any Hamiltonian path from $x$ to $y$ that contains all edges of $P$ must avoid vertex $u$.

In a graph \( G \), matching $M$ is a set of edges where no two edges share a common vertex. When a matching covers each vertex of \( G \) incident with exactly one edge in $M$ is a perfect matching. Ruskey and Savage \cite{RS93} ask the question: In the \( n \)-dimensional hypercube \( Q_n \) for \( n \geq 2 \), does any matching extend to a Hamiltonian cycle? Dvořák \cite{DT05} proved that for $n\geq 2$, the induced subgraph of a linear forest containing at most $2n-3$ edges in the $n$-dimensional hypercube can extend into a Hamiltonian cycle. Moreover, Caha and Koubek \cite{11} showed the bound is sharp. Kreweras \cite{Kre96} conjectured that any perfect matching of \( Q_n \) for \( n \geq 2 \) can be extended to a Hamiltonian cycle. Fink \cite{Fin07, Fink09}, resolve this conjecture and show that Ruskey and Savage's problem holds for \( n \in \{2, 3, 4 \}\). Later, Wang and Zhao \cite{WC18} proved for hypercube \( Q_5 \).

Gregor \cite{Gre09} strengthened Fink's result by splitting the hypercube $Q_n$ into subcubes of non-zero dimensions within the hypercube; a perfect matching can be extended to a Hamiltonian cycle on these subcubes if and only if the perfect matching connects them. Similarly, Wang and Zhang \cite{WZ16} proved that Ruskey-Savage problem holds for small matchings and obtained the result below.
\begin{theorem}[\cite{WZ16}] \label{wz16}
Consider $M$ is a set of matchings in \( Q_n \) for \( n \geq 4 \), such that \( |M| \leq 3n - 10 \). Then, $M$ can be extended to a Hamiltonian cycle.
\end{theorem}
Dvořák and Gregor \cite{DM07} proved that for any two vertices \( x \) and \( y \) in \( Q_n \), where the vertices  $x$ and $y$ are from the opposite partite set, and given a set of up to \( 2n-4 \) prescribed edges that satisfy the natural necessary condition, it is possible to extend to a Hamiltonian path between \( x \) and \( y \) that contains the prescribrd edges set. This result is valid for all \( n \geq 2 \), except in two forbidden configurations that occur when \( n = \{3, 4 \}\).
\begin{theorem}[\cite{DM07}] \label{n1}
Let $P\subseteq E(Q_n)$ and $x, y \in V(Q_n)$ for $n\geq 5$,  be two vertices of distinct partite sets that satisfy natural necessary conditions such that $|P|\leq 2n-4$. Then, a Hamiltonian path exists joining $x$ and $y$ passing through $P$ if and only if the induced subgraph of $P$ is a linear forest.
\end{theorem}
In this paper, we consider a set of matching $M$ such that $xy\notin M$, and show that matching $|M| \leq 3n - 13$ for \( n \geq 5 \) can be extended to a Hamiltonian path in  \( Q_n  \), between every two vertices $x$ and $y$ from distinct partite sets.
\begin{theorem}\label{main}
Consider a set of matchings $|M| \leq 3n - 13$ in \( Q_n \) for \( n \geq 5 \), and $x, y$ are two vertices from distinct partite set with $xy\notin M$. Then, a Hamiltonian path exists in $Q_n$ containing $M$ joining $x$ and $y$.
\end{theorem}
The structure of this paper is as follows: Section \ref{sec2} introduces the notation and some preliminary results used in the proof of our main result, while Section \ref{sec3} presents the main results.
\section{Preliminaries and lemmas}\label{sec2}
Notation and terminology used below, if not explicitly defined here, can be found in \cite{bondy}. The edge set and vertex set in a graph \( G \), denoted by \( E(G) \) and \( V(G) \), respectively. Consider  the two subgraphs \( H \) and \( H'\)  of \( G \). The  \( H + H' \) represent the graph with edge set \( E(H) \cup E(H') \) and vertex set \( V(H) \cup V(H') \).  The set of vertices incident to the edges in \( P \),  denoted by \( V(P) \), where a subset \( P \subseteq E(G) \). The graph  \( H + P \) with vertex set \( V(H) \cup V(P) \) and edge set \(E(H) \cup P \). When removing the edges set  \( P \subseteq E(G) \)  from \( G \), the resultant graph is denoted as \(G-P \). Similarly, a subset \( S \subseteq V(G) \), \(G-S \) describes the graph obtained by deleting all vertices in \( S \) and all edges incident to those vertices. When  \( P = \{e\} \) or \( S = \{x\} \), the notations \( G - P \), \( G - S \),  \( H + P \),  and \( V(P) \) are simplified to  \( G - e \), \( G - x \), \( H + e \),  and \( V(e) \), respectively.

In the hypercube graph $Q_n$ for $[n]=\{1, 2, \dots, n\}$  with vertex set $V(Q_n) = \{ u : u = u^1 \dots u^n\}$  and  $u^i \in \{0, 1\}$ for $i \in [n]\}$ and edge set $E(Q_n) = \{xy : |\Delta(x, y)| = 1 \}$, where $\Delta(x, y) = \{ i \in [n] : x^i \neq y^i\}$. The $dim(xy)$ is dimension of an edge $xy \in E(Q_n)$ is the integer $j$ with $\Delta(x, y) = \{ j \}$. The exact $j$-dimension edges form a layer, defined as $E_j$, in $Q_n$. Note that the edge set $E(Q_n)$ is divided into $n$ layers, each having $2^{n-1}$ edges.

The parity $p(u)$ of a vertex $u$ in $Q_n$ is described by $p(u) = \sum_{i=1}^n u^i(\text{mod 2})$. Hence, there are \( 2^{n-1} \) vertices of parity 0 and \( 2^{n-1} \) vertices of parity 1  in \( Q_n \). The vertices with parity 0 and 1  are white and black, respectively. Since \( Q_n \) is a bipartite graph, the color classes form a bipartition of \( Q_n \), meaning that \( p(x) = p(y) \) if and only if \( d(x, y) \) is even. A set of pairs \( \{ \{ a_i, b_i \} \}_{i=1}^k \), where \( a_i \) and \( b_i \) are different vertices of \( Q_n \), is considered balanced if the number of white and black vertices are equal. In the below proofs, we generally use that \( Q_n \) is bipartite.

For every $j \in [n]$, consider $Q_{n-1, j}^{0}$ and $Q_{n-1, j}^{1}$ denote two $(n-1)$-dimensional subcubes of a hypercube $Q_n$, where the $j$ superscript be omitted if the context is clear. These subcubes are by vertices of $Q_n$ where the $j$-th coordinate is fixed at 0 or 1, respectively. Note that $Q_n - E_j = Q_{n-1}^1 \cup Q_{n-1}^0$, and we say $Q_n$ is divided into the two $(n-1)$-dimensional subcubes $Q_{n-1}^0$ and $Q_{n-1}^1$ by $E_j$. For $\alpha \in \{0, 1\}$, any vertex $u_\alpha \in V(Q_{n-1}^\alpha)$ has a unique neighbor in $Q_{n-1}^{1-\alpha} $, denoted as $u_{1-\alpha}$. Similarly, for any edge $uv=e \in E(Q_{n-1}^\alpha)$, the corresponding edge in $Q_{n-1}^{1-\alpha}$ is $e_{1-\alpha}$ denotes the edge  $u_{1-\alpha}v_{1-\alpha}\in E(Q_{n-1}^{1-\alpha})$.
\begin{lemma}[\textup{\cite{T07}}]\label{t07}
Consider two disjoint edges $e_1$ and $e_2$ in $Q_n $ where $n \geq 2$. Then, $Q_n$ can be split into two $(n-1)$-dimensional subcubes, such that one contain $e_1$ and the other $e_2$.
\end{lemma}
A path is a \( u,v \)-path connecting the vertices \( v \) and \( u \) and denoted as \( P_{uv} \). A spanning subgraph of a graph \( G \) consisting of disjoint \( k \)  paths is known as a spanning \( k \)-path of \( G \). In particular, when \( k = 1 \), the spanning 1-path is the Hamiltonian path. If \( e \in E(P) \) then,  path \( P \) is containing the edge \( e \).

Note that the next classical result in \cite{H84} was obtained by Havel.
\begin{theorem}[\textup{\cite{H84}}]\label{h84}
Consider two vertices $x$ and $y$ in $Q_n$ with $p(y)\neq p(x)$. Then, $Q_n$ has a Hamiltonian path connecting $x$ and $y$.
\end{theorem}
\begin{lemma}[\textup{\cite{DT05}}]\label{a2}
Consider $x,y \in V(Q_n)$, where $n\geq 2$, and $f \in E(Q_n)$ with $xy\neq f$ and $p(y)\neq p(x)$. Then, a Hamiltonian path $P_{xy}$ exists in $Q_n$ containing $f$.
\end{lemma}
\begin{lemma}[\textup{\cite{WZ16}}]\label{a4}
Consider $e_1$ and $e_2$ be two disjoint edges in $Q_n$ where $n\geq 4$, and $u, v \in V(Q_n)$ such that $p(u) \neq p(v)$ and $uv \notin \{ e_1, e_2 \}$. Then, a Hamiltonian path exists in $Q_n$ joining $u$ and $v$ containing $e_1$ and $e_2$.
\end{lemma}
\begin{theorem}[\textup{\cite{DT05}}]\label{dt}
Consider pairwise distinct vertices $u, v, x, y$ of $ Q_n$ where $n \geq 2$, with $p(u) \neq p(v)$ and $p(x) \neq p(y)$. Then, (i) a spanning $2$-path $P_{uv} + P_{xy}$ exists in $Q_n $ moreover, (ii) the path $P_{uv}$ can be chosen such that $P_{uv} = uv$ if $d(u, v) = 1$ unless $n = 3$, $d(x, y) = 1$, and $d(xy, uv) = 2$.
\end{theorem}
\begin{lemma}[\textup{\cite{WZ16}}]\label{a1}
Consider \( xy \) and \( f \) are two disjoint edges in \( Q_n \) where \( n \geq 5 \), and consider \( u, v \in V(Q_n) \setminus \{x, y\} \) with \( p(u) \neq p(v) \) and \( uv \neq f \). Then, a Hamiltonian path exists in \( Q_n - \{x, y\} \) connecting \( u \) and \( v \) containing \( f \).
\end{lemma}
\begin{lemma}[\textup{\cite{CR}}]\label{tt}
Consider $u, v, x, y$ are pairwise distinct vertices of $Q_n$ where $n\geq 4$, with $p(u) = p(v) \neq p(x) = p(y)$. Then, a spanning 2-path $P_{uv} + P_{xy}$ exists in $Q_n$.
\end{lemma}
\begin{theorem}[\textup{\cite{GP08}}]\label{a8}
Consider a balanced pair set $\{(a_i, b_i)\}_{i=1}^k$  in \( Q_n \) where \( n \geq 1 \), with $2k - | \{a_i b_i\}_{i=1}^k \cap E(Q_n)|< n$. Then, a spanning $k$-path \( \sum_{i=1}^{k}P_{a_i b_i} \) exists of \( Q_n \).
\end{theorem}
\begin{lemma}[\textup{\cite{WZ16}}]\label{a5}
Consider \( e \in E(Q_n) \) where \( n \geq 5 \), and \( u, v, x, y \in V(Q_n) \) are distinct vertices such that \( p(u) \neq p(v) \), \( p(x) \neq p(y) \), and \( \{x, y\} \cap V(e) = \emptyset \). Then, a spanning 2-path \( P_{uv} + P_{xy} \) exists in \( Q_n \) containing \( e \).
\end{lemma}
\begin{lemma}\label{a6}
Consider \( e_1, e_2 \in E(Q_n) \) and \( u, v, x, y \in V(Q_n) \) where \( n \geq 6 \), be distinct vertices such that \( p(u) \neq p(v) \), \( p(x) \neq p(y) \), and \( \{x, y\} \cap V(e_1) = \emptyset \) also \( \{x, y\} \cap V(e_2) = \emptyset \). Then, a spanning 2-path \( P_{uv} + P_{xy} \) exists in \( Q_n \) containing \( e_1 \) and $e_2$.
\end{lemma}
\begin{proof}
If \(uv = e_1\) or $uv = e_2$, Lemma \ref{a1} implies a spanning 2-path $P_{uv} = uv$ and $P_{xy}$ exists in \(Q_n\). Thus, the result holds. Now, we proceed under the assumption that \(uv\neq e_1\) or $uv\neq e_2$. W.l.o.g., suppose that \(v \notin V(e_1)\) and \(v \notin V(e_2)\) and \(p(v) = p(y) \neq p(u) = p(x)\). Since \(d(u, x) \geq 2\), there must be \(j \in [n]\) with \(u \in V(Q^0_{n-1})\) and \(x \in V(Q^1_{n-1})\), such that \(e_1, e_2 \notin E_j\). By symmetry we consider $e_1, e_2 \in E(Q_{n-1}^0)$, otherwise, \(e_1 \in E(Q^0_{n-1})\) and \(e_2 \in E(Q^1_{n-1})\).  
	
First, consider the case where \(y \in V(Q^0_{n-1})\). Since \(p(u) \neq p(y)\) and $uy \neq e_1$ and $uy \neq e_2$, apply Lemmas \ref{a4} or \ref{a2}, a Hamiltonian path \(P^0_{uy}\) exists containing $e_1$ and, $e_2$. If \(v \in V(Q^0_{n-1})\), let \(w_0\) be the neighbor of \(v\) such that \(w_0 \in V(P^0_{uy}[v, y])\). As \(v \notin V(e_1)\) and \(v \notin V(e_2)\), we have \(vw_0 \neq e_1\) and \(vw_0 \neq e_2\). Since \(p(w_1) = p(v) \neq p(x)\), Lemmas \ref{h84} or Lemma \ref{a2} when $e_1 \in V(Q_{n-1}^1)$, a Hamiltonian path \(P^1_{w_1x}\) exists in \(Q^1_{n-1}\) containing $e_1$. Thus spanning 2-path of \(Q_n\) is $P^0_{uy}[u, v],  P^0_{uy}[y, w_0] + \{w_0w_1\} + P^1_{w_1x}$. If \(v \in V(Q^1_{n-1})\), since \(|E(P^0_{uy}) \setminus \{e_1, e_2\}| = 2^{n-1} - 3 > 7\) for \(n \geq 6\), there exists an edge \(s_0t_0 \in E(P^0_{uy}) \setminus \{e_1, e_2\}\) with \(\{s_1, t_1\} \cap \{x, v\} = \emptyset\). Assume \(s_0\) lies on \(P^0_{uy}[u, t_0]\), indicating \(s_0\) is closer to \(u\) than \(t_0\) on \(P^0_{uy}\). Since \(p(s_1) \neq p(t_1)\), \(p(v) \neq p(x)\) and \(n - 1 \geq 5\) there must be \(p(s_1) \neq p(v)\), Lemmas \ref{a5} or \ref{dt} ensures a spanning 2-path \( P^1_{s_1v} + P^1_{t_1x} \) in \(Q^1_{n-1}\). The desired spanning 2-path of \(Q_n\) is $ P^0_{uy}[u, s_0] + \{s_0s_1\} + P^1_{s_1v} + P^0_{uy}[y, t_0] + \{t_0t_1\} + P^1_{t_1x}$.
	
Next, consider the case where \(y \in V(Q^1_{n-1})\). If \(v \in V(Q^0_{n-1})\), Lemmas \ref{a4} or \ref{a2} or Theorem \ref{dt}, Hamiltonian paths \(P^0_{uv}\) in \(Q^0_{n-1}\) , and \(P^1_{xy}\) in \(Q^1_{n-1}\) exist passing through \(e_1\) and $e_2$. The spanning 2-path is \( P^0_{uv} + P^1_{xy} \) in \(Q_n\).  If \(v \in V(Q^1_{n-1})\), since $|\{w_1 \in V(Q^1_{n-1}) \setminus \{x\} : p(w_1) = p(x)\}| \geq 2^{n-2} - 1 > 1$, there exists a vertex \(w_1 \in V(Q^1_{n-1}) \setminus \{x\}\) with \(p(w_1) = p(x)\) and \(uw_0 \neq e_1, e_2\). Note that \(p(u) = p(x) = p(w_1) \neq p(v) = p(y) = p(w_0)\). By Lemmas \ref{a2} or \ref{a4}, a Hamiltonian path \(P^0_{uw_0}\) exists in \(Q^0_{n-1}\) passing through $\{e_1, e_1\}$, and by Lemma \ref{a5} or Theorem \ref{dt}, a  spanning 2-path \( P^1_{w_1v} + P^1_{xy} \) of \(Q^1_{n-1}\) containing $e_1$. The desired spanning 2-path of \(Q_n\) is $P^0_{uw_0} +\{w_0w_1\} + P^1_{w_1v}$ and $P^1_{xy}$.
\end{proof}
\begin{lemma}\label{a7}
Consider \( e \in E(Q_n) \) where \( n \geq 6 \), and \( u, v, x, y, w, z \in V(Q_n) \) be pairwise distinct vertices such that \( p(u) \neq p(v) \), $p(x) \neq p(y)$,  $p(w) \neq p(z)$, and \( \{x, y, w, z\} \cap V(e) = \emptyset \). Then, (i) a spanning 3-path \( P_{uv} + P_{xy} + P_{wz} \) exists in \( Q_n \) containing \( e \). (ii) moreover, if $uv= e$ the path $P_{uv}$ can be chosen such that $P_{uv} = uv$.
\end{lemma}
\begin{proof}
Given that $u, v, x, y, w, z$ are pairwise distinct vertices with $p(u) = p(x) = p(w)\neq p(v)=p(y)=p(z)$. Since \(d(u, x) \geq 2\) there exists \(j \in [n]\) such that \(u \in V(Q^0_{n-1})\), \(x \in V(Q^1_{n-1})\), and \(e \notin E_j\). If \(u \in V(e)\) or $uv = e$, then \(e \in E(Q^0_{n-1})\). Otherwise, if \(u \notin V(e)\), we can assume w.l.o.g., \(e \in E(Q^0_{n-1})\).  
	
First, consider the case where \(v, w, z \in V(Q^0_{n-1})\). When $uv = e$ then by symmetry $uv \in E(Q^0_{n-1})$, by Theorem \ref{dt} or when $uv \neq e$ and $\{w, z\}\cap V(e)= \emptyset$ by Lemma \ref{a5}, there exists a spanning 2-path \(P^0_{uv} + P^0_{wz}\) passing through $e$. If \(y \in V(Q^0_{n-1})\), let $y \in P_{wz}^0$ and \(s_0\) be the neighbor of \(y\). Vertex $s_0$ is adjacent to $s'_0$ and $r_0$ to $y$, both are on same side of \(P^0_{wz}\) near to $w, z$ both $yr_0\neq e$ and $s_0s'_0\neq e$ . On other side $y$ has neighbor $r'_0$ with $p(x)\neq p(r'_1)$ and $p(r_1)\neq p(s'_1)$. By Theorem \ref{dt}, there exists a spanning 2-path \(P^1_{r'_1x} + P^1_{s'_1r_1}\) in \(Q^1_{n-1}\). Hence, spanning 3-path of \(Q_n\) is $P^0_{uv},  P^0_{wz} + \{r_0r_1, s'_0s'_1\} + P_{r_1s'_1}^1 - P^0_{wz}[y, s_0]$ and  $P^1_{r'_1x} + \{ys_0, r'_0r'_1\} + P^0_{wz}[r'_0, s_0]$, shown in Figure \ref{Fig.1}(a). If \(y \in V(Q^1_{n-1})\),  Theorem \ref{h84}, there exists a Hamiltonian path \( P^1_{xy} \) in \(Q^1_{n-1}\). The desired spanning 3-path of \(Q_n\) is $ P^0_{uv} + P^0_{wz} + P^1_{xy}$ passing through $e$.
\begin{figure}[H]
	\centering
	\includegraphics[width=1.0\linewidth, height=0.26\textheight]{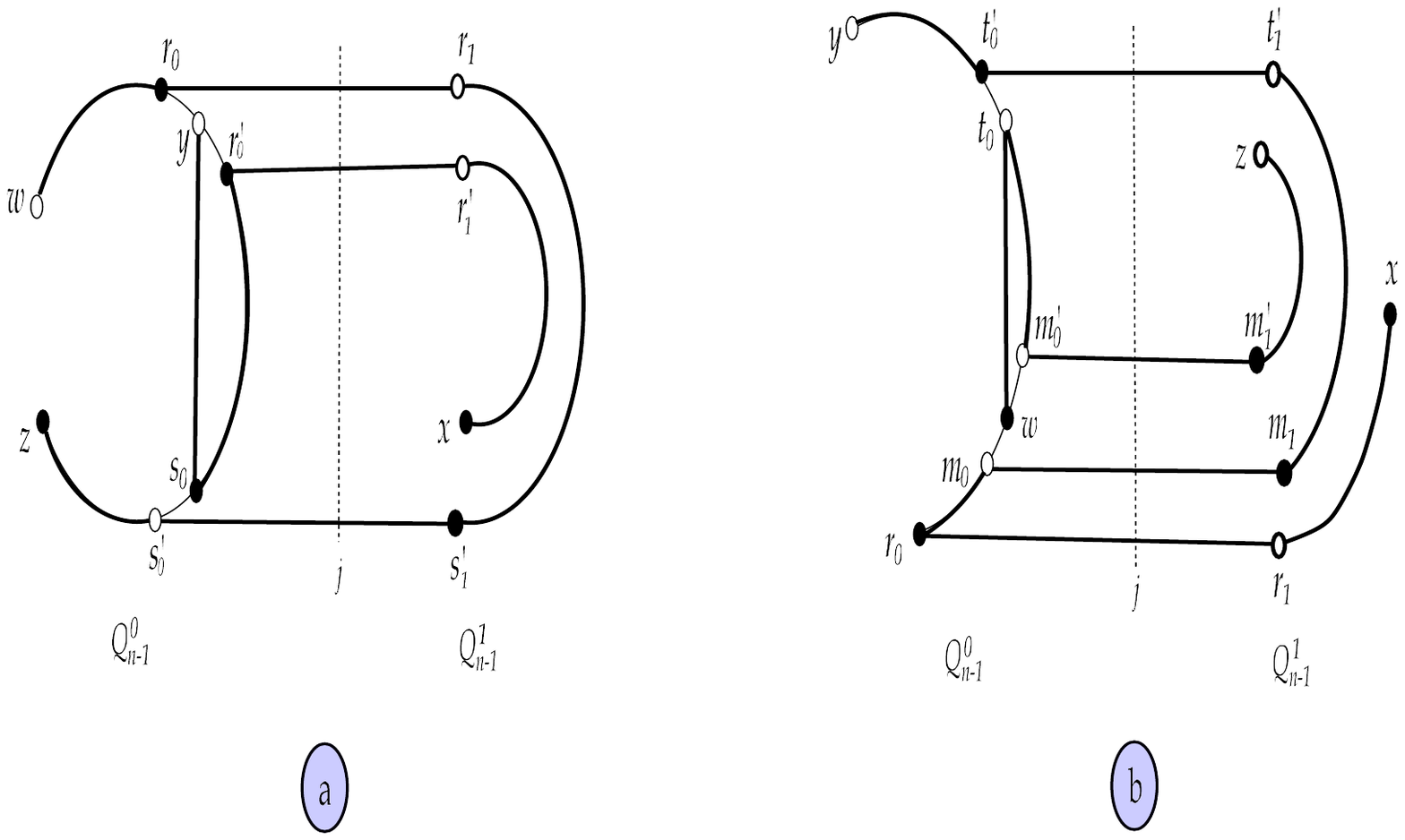}
	\caption[A]{\small Depiction of $Q_n$ into spanning 3-path extending $e$.}\label{Fig.1}
\end{figure}

Next, consider the case where \(v, z \in V(Q^1_{n-1})\) such that $uv \neq e$. Choose neighbors $s_0, r_0$ of $u$ and $y$ respectively, such that $p(r_1)\neq p(v)$ and $p(s_1)\neq p(x)$. By Lemma \ref{a5}, there exists a spanning 2-path \(P^0_{us_0} + P^0_{yr_0}\) in \(Q^0_{n-1}\) passing through \(e\). Let $w \in V(Q_{n-1}^0)$, consider $w \in P_{yr_0}^0$ and \(t_0\) be the neighbor of \(w\) such that $t_0$ is adjacent to $t'_0$ and $m_0$ to $w$, both are on same side of \(P^0_{yr_0}\) near to $y, r_0$. On other side $w$ has neighbor $m'_0$ with $p(z)\neq p(m'_1)$ and $p(m_1)\neq p(t'_1)$. Since $x, r_1, v, s_1, z, m'_1, m_1, t'_1 $ are distinct. For $n-1\geq 6$, by Theorem \ref{a8}, there exists a spanning 4-path \( P^1_{r_1x} + P^1_{s_1v} + P^1_{m'_1z} + P^1_{t'_1m_1}\) in \(Q^1_{n-1}\). Hence, spanning 3-path of \(Q_n\) is $P^0_{us_0} + s_0s_1 + P^1_{vs_1},  P^0_{yr_0} + \{m_0m_1, t'_0t'_1\} + P_{m_1t'_1}^1 - P^0_{yr_0}[w, t_0],  P^1_{m'_1z} + \{wt_0, m'_0m'_1\} + P^0_{wz}[m'_0, t_0]$ containing $e$, depicted in Figure \ref{Fig.1}(b). If \(w \in V(Q^1_{n-1})\), note that \(p(x) = p(v) = p(z) \neq p(r_1) = p(s_1) = p(w)\). By Theorem \ref{a8}, there exists a spanning 3-path \(P^1_{xr_1} + P^1_{vs_1} + P^1_{wz}\) in \(Q^1_{n-1}\). Hence,  spanning 3-path \(P^0_{us_0} + s_0s_1 + P^1_{vs_1}, P^0_{yr_0} + r_0r_1 + P^1_{xr_1}, P^1_{wz} \) of \(Q^1_{n-1}\) containing $e$.
\end{proof}
\begin{lemma}[\textup{\cite{WZ16}}]\label{wz}
Consider $M$ be a matching of $Q_n$, where $n \geq 4$, with $|M| \leq 2n - 8$, and let $ux$ and $vy$ are two disjoint edges in $E(Q_n)$ satisfying $\{u, v\} \cap V(M) = \emptyset$ and $xy \notin M$. Then, a spanning $2$-path $P_{vy} + P_{ux}$ exists of $Q_n$ that containing $M$.
\end{lemma}
\begin{lemma}\label{2.11}
Consider $M$ be a matching of $Q_n$, where $n \geq 5$, with $|M| \leq 2n - 10$, and $u, x \in V(Q_n)$ such that $p(u)\neq p(x)$. Let $vy$ and $wz$ be two disjoint edges in $E(Q_n)$, satisfying $\{u, v, w\} \cap V(M) = \emptyset$ and $\{xy, yz, zx\} \notin M$. Then, a spanning $3$-path $P_{ux} + P_{vy} + P_{wz}$ exists of $Q_n$  containing $M$.
\end{lemma}
\begin{proof} 
We will prove this lemma by induction on $n$. For the base case $n=5$ since $M= \emptyset$, the lemma holds by Theorem \ref{a8}. Assume that for $n - 1 \ \geq 5$ the lemma holds. We must show that it is true for $ n\geq 6$. Choose $j \in [n]$ with given condition $| (M \cup \{vy, wz\}) \cap E_j |$  reduced to the smallest possible size. Since $| M \cup \{vy, wz\} | \leq 2n - 8$, it follows that $| (M \cup \{vy, wz\}) \cap E_j | \leq 1$. If $| (M \cup \{vy, wz\}) \cap E_j | = 1$, given $| (M \cup \{vy, wz\}) \cap E_i | \geq 1$ for each $i \in [n]$, at least eight options exists for $j$. Hence, $j$ can be chosen with $(M \cup \{vy, wz\}) \cap E_j = M \cap E_j := \{s_0s_1\}$ and $\{x, y, z\} \cap \{s_0, s_1\} = \emptyset$. Next, decompose $Q_n$ into $Q^0_{n-1}$ and $Q^1_{n-1}$ by $E_j$. Since $\{vy, wz\} \subseteq E(Q^0_{n-1}) \cup E(Q^1_{n-1})$, and using symmetry, suppose that $vy \in E(Q^0_{n-1})$.
	
Case 1. \(wz \in E(Q^0_{n-1}) \) and $u, x \in V(Q_{n-1}^0)$.

Subcase 1.1. $|M_0| \leq 2n-12 = 2(n-1) - 10$.
	
By induction hypothesis, a spanning $3$-path \( P^0_{ux} + P^0_{vy}+ P^0_{wz} \) exists in \( Q^0_{n-1} \) passes through \( M_0 \).
	
Assume \( M_c = \{s_0s_1\} \). W.l.o.g., let \( s_0 \in V(P_{ux}) \), choose a neighbor \( r_0 \) of \( s_0 \) on \( P^0_{ux} \). Since \( M \) is a matching and \( s_0s_1 \in M \), it follows that \( \{s_0r_0, s_1r_1\} \cap M = \emptyset \). Now, if \( M_c = \emptyset \), w.l.o.g., we may assume that \( |E(P^0_{ux})| \geq |E(P^0_{vy})| \geq |E(P^0_{wz})| \). This gives \( |E(P^0_{ux})| \geq \frac{2^{n-1}}{3} - 1 \). Since \( |E(P^0_{ux}) \setminus M_0| - |M_1| \geq \frac{2^{n-1}}{3} - 1 - (2n - 10) > 3 \) for \( n \geq 6\), we can choose an edge \(s_0r_0 \in E(P^0_{ux}) \setminus M_0 \) such that \( s_1r_1 \notin M_1 \). In both cases, \( M_1 \leq 2(n-1) - 4 \) edges. Using Theorem \ref{n1}, a Hamiltonian path \( P^1_{s_1r_1} \) exists in \( Q^1_{n-1} \) that containing \( M_1 \). Thus, the spanning $3$-path of \( Q_n \) is $P_{ux} := \langle P^0_{ux} + P^1_{s_1r_1} + \{s_0s_1, r_0r_1\} - s_0r_0 \rangle, P_{vy} := P^0_{vy}$ and  $P_{wz} := P^0_{wz}$.
	
 Subcase 1.2. $|M_0| \geq 2n - 11$. Now, $|M_c| \leq 1$.
	
Subcase 1.2.1: $M_1 = \emptyset$.
	
In this case, $M_0 \cup \{vy, wz\}$ forms a linear forest containing no more than $2(n-1) - 8$ edges. Since $n-1\geq 5$, according to Theorem \ref{n1}, a Hamiltonian path $P^0_{ux}$ exists in $Q^0_{n-1}$ that containing $M_0 \cup \{vy, wz\}$. Removing $\{vy, wz\}$ from $P^0_{ux}$ splits it into three disjoint paths, $P^0_u, P^0_x$ and $P^0_y$, ending at $u, x$ and $y$, respectively. If $M_c= \{s_0s_1\}$, and $\{u, v, w, x, y, z\} \cap \{s_0, s_1\} = \emptyset$, w.l.o.g., we can assume $s_0 \in V(P^0_u)$. Let $t_0$ be a neighbor of $s_0$ on $P^0_u$. Since $s_0s_1 \in M$ and neither $\{xy, yz, zx\}$ nor $u, v, w$ are in $M$. Let $m_0r_0 \notin M_0$  on $P^0_x$, and $|P^0_{wy}| \geq |P^0_{vx}| \geq |P^0_{uz}|$ such that $(E(P^0_{ux})\setminus (M_0 + \{vy, wz\}) = 2^{n-1} - (2n-10) =30$ for $n\geq 6$, each direction has at least $5$-edges not in $M_0$. Then there exist edges $a_0b_0 \in E(P^0_y) - M_0$ in same direction as $s_0t_0$ and $c_0d_0 \in E(P^0_y) - M_0$  same direction as $m_0r_0$. If $M_c = \emptyset$, with $\{u, v, w\} \cap V(M) = \emptyset$ and $\{xy, yz, zx\} \notin M$, edges $s_0t_0 \in E(P^0_u) - M_0$, $m_0r_0 \in E(P_x) - M$ and $\{a_0b_0, c_0d_0\} \in E(P^0_y) - M_0$ exists.
	
Assume $t_0 \in V(P^0_u[z, s_0])$, $m_0 \in V(P^0_x[x, r_0])$, $d_0 \in V(P^0_y[y, c_0])$ and $b_0 \in V(P^0_y[y, a_0])$. There should be a direction $i\in [n-1]$ in $Q_{n-1}^1$ such that $\{c_1d_1, r_1m_1\} \in E(Q_{n-2}^{1L})$ and $\{s_1t_1, a_1b_1\} \in E(Q_{n-2}^{1R})$. Since $p(a_1) \neq p(b_1)$, $p(t_1) \neq p(s_1)$, $p(c_1) \neq p(d_1)$ and $p(m_1) \neq p(r_1)$, using Lemmas \ref{dt} in case whether $p(a_1) \neq p(t_1)$, $p(c_1) \neq p(m_1)$ or Lemma \ref{tt} in case $p(a_1) = p(t_1)$, $p(c_1) = p(m_1)$ there is a  spanning $2$-path $P_{a_1t_1}+ P_{s_1b_1}$ in $Q^{1R}_{n-2}$ and $P^1_{c_1m_1}+ P^1_{d_1r_1}$ in $Q^{1L}_{n-2}$. Thus, the spanning $3$-path of $Q_n$ is  $P_{ux} := \langle P^0_u[u, s_0] + P^0_x[x, m_0] + \{s_0s_1, b_0b_1, c_0c_1,m_0m_1\} + P^1_{s_1b_1} + P^1_{c_1m_1} + P^0_y- \{P^0_y[y, d_0]+ P^0_y[w, a_0]\} \rangle$, $P_{vy} := \langle P^0_y[y, d_0] + P^0_x- P^0_x[x, m_0] + \{r_0r_1, d_0d_1\}+ P^0_{r_1d_1}\rangle $ and $P_{wz} := \langle P^0_y - P^0_y[y, a_0]+ \{a_0a_1, t_0t_1\}+ P^0_u - P^0_u[u, s_0] + P^1_{a_1t_1}\rangle$.
	
Subcase 1.2.2: $|M_1| = 1$.
	
Now, $|M_0| \geq 2n - 11$ and $|M| \leq 2n - 10$ imply $|M_0| = 2n - 11$ and $M_c = \emptyset$. For $n \geq 6$, $|M_0| = 2n - 11 \geq 1$ there exists $s_0t_0 \in M_0$ such that $\{x, y, z\} \cap \{s_0, t_0\} = \emptyset$. Applying induction a spanning $3$-path $P^0_{ux}+ P^0_{vy}+P^0_{wz}$ exists in $Q^0_{n-1}$ passing through $M_0 \setminus \{s_0t_0\}$ . If $s_0t_0 \in E(P^0_{ux}) \cup E(P^0_{vy}) \cup E(P^0_{wz})$, the lemma follows, if $s_0t_0 \notin E(P^0_{ux}) \cup E(P^0_{vy}) \cup E(P^0_{wz})$ w.l.o.g., we may assume $s_0 \in V(P^0_{ux})$.
	
If $t_0 \in V(P^0_{ux})$, consider neighbors $m_0, r_0$ of $s_0, t_0$ on $P^0_{ux}$, respectively, such that exactly one of $m_0$ and $r_0$ lies on $P^0_{ux}[s_0, t_0]$ ensuring $m_1r_1 \notin M_1$. Since $\{u, x\}\cap \{s_0, t_0\}=\emptyset$ which is always possible. Since $p(m_1)\neq p(r_1)$, using Lemma \ref{a2} provides a Hamiltonian path $P^1_{m_1r_1}$ passes through $M_1$ in $Q^1_{n-1}$. The spanning $3$-path is $P_{ux} := \langle P^0_{ux} + P^1_{m_1r_1} + \{s_0t_0, s_0s_1, r_0r_1\} - \{s_0m_0, r_0t_0\}\rangle$,	$P_{vy} := P^0_{vy}$ and $P_{wz} := P^0_{wz}$.

If $t_0 \in V(P^0_{vy})$, assume $|E(P^0_{ux})| \geq |E(P^0_{vy})| \geq |E(P^0_{wz})|$, then $|E(P^0_{ux})| \geq \frac{2^{n-1}}{3} - 1$. Select $b_0 \in V(P^0_{ux}) - s_0$ with $p(b_0) = p(s_0)$ and $b_0 \notin V(M_0)$. For $n \geq 6$,  $|\{ b_0 \in V(P^0_{ux}) - s_0\} : p(b_0) = p(s_0)\}| \geq \lceil\frac{2^{n-2}}{3}\rceil  - 1 \geq (2n - 11) + 3$ ensures that there are at least three distinct options for selecting such a vertex \( b_0 \). Let \( a_0 \) denote the neighbor of \( b_0 \) such that \( a_0 \in V(P^0_{ux}[b_0, s_0]) \). Furthermore, we can select \( b_0 \) in a way that ensures \( a_1 \notin V(M_1) \). Since \( b_0 \notin V(M_0) \), it follows that \( b_0a_0 \notin M_0 \). Define \( r_0 \) as the neighbor of \( s_0 \) on \( P^0_{ux} \), where \( s_0 \) lies within \( P^0_{ux}[r_0, a_0] \). Similarly, let \( m_0 \) be a neighbor of \( t_0 \) on \( P^0_{vy} \) such that \( m_1 \notin V(M_1) \). Given that \( \{u, v, x, y\} \cap \{s_0, t_0\} = \emptyset \), this is always possible. In \( Q_{n-1}^1 \), we find that \( p(r_1) \neq p(b_1) \), \( p(m_1) \neq p(a_1) \), and \( \{s_1, a_1\} \cap V(M_1) = \emptyset \). For \( n - 1 \geq 6 \), Lemma \ref{a5} ensure the existence of a spanning $2$-path \( P^1_{r_1b_1} + P^1_{m_1a_1} \) in \( Q_{n-1}^1 \), passing through \( M_1 \). Since \( M \) is a matching and \( s_0t_0 \in M \), we conclude \( \{s_0r_0, t_0m_0\} \cap M = \emptyset \). Thus, we define the spanning $3$-path of \( Q_n \) as $P_{ux} := \langle P^0_{ux} - P^0_{ux}[r_0, b_0] + \{r_0r_1, b_0b_1\} + P^1_{r_1b_1} \rangle, P_{vy} := \langle P^0_{vy} + P^0_{ux}[s_0, a_0] + P^1_{m_1a_1} + \{s_0t_0, a_0a_1, m_0m_1\} - m_0t_0 \rangle$ and  $ P_{wz}:= P^0_{wz}$ Similarly, we can prove when $t_0\in P^0_{wz}$. 
	
Case 2. Assume $x \in V(Q_{n-1}^1)$.  
	
Subcase 2.1. $|M_0| \leq 2n - 12 = 2(n-1) - 10$.
	
Subcase 2.1.1. $|M_c| = 0$.
	
Since \( M_0 \leq 2(n - 1) - 10 \), choose $r_0 \in V(Q_{n-1}^0)$ be neighbor of $u$ such that $p(r_1)\neq p(x)$. According to induction hypothesis, a spanning 3-path \( P_{vy}^0 + P_{wz}^0 + P^0_{ur_0} \) exists in \( Q^0_{n-1} \) passing through $M_0$. Now, \( M_1 \leq \frac{2n-10}{2}< 2(n-1)-4\), using Theorem \ref{n1}, a Hamiltonian path  $P^1_{xr_1}$ exists containing \( M_1 \). Thus, $ P_{ux}:=P^0_{ur_0} + r_0r_1 + P^1_{xr_1}$, $P_{vy}^0$ and $ P^0_{wz} $ is the spanning $3$-path of \( Q_n \) containing $M$.  
	
Subcase 2.1.2. $|M_c|= 1$. Let $M_c=\{s_0s_1\}$
	
If \( M_c = \{s_0s_1\} \), given that \( \{u, v, w\} \cap V(M) = \emptyset \) and \( \{x, y, z\} \cap \{s_0, s_1\} = \emptyset \). Choose $r_0 \in V(Q_{n-1}^0)$ be neighbor of $u$ such that $p(r_1)\neq p(x)$. Since \( M_0 \leq 2(n - 1) - 10 \) by symmetry, assume \( |M_0| \geq |M_1| \). By induction hypothesis, a spanning $3$-path \( P^0_{ur_0}+P^0_{vy}+ P^0_{wz} \) of \( Q^0_{n-1} \) exists that contains \( M_1 \). Let $s_0 \in V(P^0_{ur_0})$ choose $t_0$ adjacent to $s_0$ on $P^0_{ur_0}$ such that $s_1t_1 \in Q_{n-1}^1$ and $|E(P^0_{wz})| \geq |E(P^0_{vy})| \geq |E(P^0_{ur_0})|$, then $|E(P^0_{wz})| \geq \frac{2^{n-1}}{3} - 2- 2n - 10 > 3$, there must be $m_0n_0 \in E(P^0_{wz})$ such $\{m_0n_0, m_1n_1\}\cap M = \emptyset$. Now, \( |M_1| \leq 2(n-1) - 10\) and  \( \{s_1, r_1, n_1\} \cap V(M_1) = \emptyset \), by induction hypothesis, a spanning 3-path $P_{s_1t_1}^1 + P^1_{xr_1} + P^1_{m_1n_1}$ exists passing through \( M_1 \) in \( Q^1_{n-1} \). Thus, $P_{ux} := P^0_{ur_0} + P^1_{s_1t_1} + P^1_{xr_1} + \{s_0s_1, t_0t_1, r_0r_1\} - s_0t_0,  P_{vy} := P^0_{vy}$ and  $P_{wz} := P^1_{wz}+\{m_0m_1, n_0n_1\}+P^1_{m_1n_1}-m_0n_1$ is the required spanning $3$-path of \( Q_n \) passing through $M$.

 Subcase 2.2. $|M_0| \geq 2n - 11$. Now, $|M_c| \leq 1$.

Since \( |M_0| \geq 2n  - 11 \), we have $|M|=2n-10$. Given that $vy, wz \in Q_{n-1}^0$ and $x \in V(Q_{n-1}^1)$, we have $|M_0|\leq 2(n-1)-8$ and $\{v, w\}\cap V(M_1)=\emptyset$, by Lemma \ref{wz}, a spanning 2-path \( P_{vy}^0 + P_{wz}^0 \) exists in \( Q^0_{n-1} \) passing through $M_0$. If $|M_c| = \emptyset$ and \(u \in V(Q^0_{n-1})\). Let $u \in P_{wz}^0$ and \(s_0\) be the neighbor of \(u\) such that $s_0$ is adjacent to $s'_0$ and $r_0$ to $u$ on $P_{wz}^0$. Both are on same side of \(P^0_{wz}\) near to $w, z$ and $s_0s'_0\notin M_0$. On other side $u$ has neighbor $r'_0$ with $p(x)\neq p(r'_1)$ and $p(r_1)\neq p(s'_1)$. Now, \( |M_1| \leq 2\) and $\{r'_1, s'_1\}\cap M_1= \emptyset$, using Lemma \ref{a6}, a spanning 2-path  $P^1_{xr'_1} + P^1_{r_1s'_1}$ exists containing \( M_1 \). Thus, $P_{ux}:= P_{wz}[u, s_0] +P^0_{xr'_1} + \{r'_0r'_1, us_0\}$, $P_{vy}^0$ and $P_{wz}:=P^1_{r_1s'_1} +\{r_0r_1, s'_0r'_1\} + P_{wz}^0- P^0_{wz}[u, s_0]$ is the spanning $3$-path of \( Q_n \) containing $M$.  

If $|M_c| = 1$, let $M_c=\{m_0m_1\}$ and $m_0\in V(P_{vy}^0)$, choose a neighbor $n_0$ on $P_{vy}^0$. Since \(u \in V(Q^0_{n-1})\), let $u \in P_{wz}^0$ and \(s_0\) be the neighbor of \(u\) such that $s_0$ is adjacent to $s'_0$ and $r_0$ to $u$ on $P_{wz}^0$. Both are on same side of \(P^0_{wz}\) near to $w, z$ and $s_0s'_0\notin M_0$. On other side $u$ has neighbor $r'_0$ with $p(x)\neq p(r'_1)$ and $p(r_1)\neq p(s'_1)$. Now, \( M_1| \leq 1\) and $\{r_1, s'_1, m_1, n_1\}\cap M_1= \emptyset$, using Lemma \ref{a7}, a spanning 3-path  $P^1_{xr'_1} + P^1_{r_1s'_1} + P_{m_1n_1}^1$ exists containing \( M_1 \). Thus, $P_{ux}:= P_{wz}[u, s_0] +P^0_{xr'_1} + \{r'_0r'_1, us_0\}$, $P_{vy}:= P_{vy}^0+ \{m_0m_1, n_0n_1\} + P_{m_1n_1}^1 - m_0n_0$ and $P_{wz}:= P_{wz}^0- P^0_{wz}[u, s_0] + P^1_{r_1s'_1} +\{r_0r_1, s'_0r'_1\}$ is the spanning $3$-path of \( Q_n \) containing $M$.  
	
Case 3: Assume \( wz \in E(Q^1_{n-1}) \) and $x, u \in V(Q_{n-1}^1)$.  
	
Since \( M_0 \) is no more than \( 2(n - 1) - 8 \) edges. According to Theorem \ref{n1}, a Hamiltonian path \( P_{vy}^0\) exists in \( Q^0_{n-1} \) passing through $M_0$. Choose $s_0r_0 \in E(P^0_{vy})$ or $s_0s_1 \in M_c$ such that $\{s_0r_0, s_1r_1\}\cap M=\emptyset$. Since \( \{u, s_1, w\} \cap V(M) = \emptyset \) with \( M_1 \leq  2(n-1)-10\) and $p(u)\neq p(x)$, by induction hypothesis, a spanning $3$-path  $P^1_{wz} + P^1_{s_1r_1}+ P_{ux}^1$ exists containing \( M_1 \). Thus, $ P_{vy}:= P_{s_1r_1}^1 + P^0_{vy} + \{s_0s_1, r_0r_1\}- s_0r_0$, $P_{ux}:= P^1_{ux}$ and $P_{wz}:= P^1_{wz} $ is the spanning $3$-path of \( Q_n \) passing through $M$.  
	
Subcase 3.1:  Assume $u \in V(Q_{n-1}^0)$.  
	
If \( M_c = \{s_0s_1\} \), given that \( \{u, v, w\} \cap V(M) = \emptyset \) and \( \{x, y, z\} \cap \{s_0, s_1\} = \emptyset \), we have \( \{u, v, w, x, y, z\} \cap \{s_0, s_1\} = \emptyset \). By symmetry, assume \( |M_0| \geq |M_1| \). Let  $u \in V(Q_{n-1}^0)$ and $vy \in E(Q_{n-1}^0)$. Choose $t_0$ neighbor of $u$ such that $p(x)\neq p(t_1)$. Since \( M_0\leq 2(n - 1) - 8 \) edges. According to Lemma \ref{wz}, a spanning 2-path \(P^0_{ut_0}+ P_{vy}^0\) exists in \( Q^0_{n-1} \) containing $M_0$.  Assume $|E(P^0_{vy})| \geq |E(P^0_{ut_0})|$, then $|E(P^0_{vy})| \geq 2^{n-2} - 1$, choose $s_0r_0 \in E(P^0_{vy})$ or $s_0s_1 \in M_c$ such that $\{s_0r_0, s_1r_1\}\cap M=\emptyset$. Since \( M_1 \leq \frac{2n-10}{2}< 2(n-1)-10\) and $p(t_1)\neq p(x)$, by induction hypothesis, a spanning $3$-path  $P^1_{wz} + P^1_{s_1r_1}+ P_{xt_1}^1$ exists containing \( M_1 \). Thus, $ P_{vy}:= P_{s_1r_1}^1 + P^0_{vy} + \{s_0s_1, r_0r_1\}- s_0r_1$, $P_{ux}:= P^0_{ut_0} + t_0t_1 + P^1_{xt_1}$ and $P_{wz}:= P^1_{wz} $ is the spanning $3$-path of \( Q_n \) passing through $M$.  	
\end{proof}
\section{Proof of Theorem \ref{main}}\label{sec3}
If \( M = \emptyset \), the result follows directly from Theorem \ref{h84}. For the remaining discussion, we consider the case when \( M \neq \emptyset \). We prove the given result by induction on \( n \). Since \( 3n - 13 \leq 2n-4 \) for \( n \leq 9 \), by Theorem \ref{n1} result holds. Now, assume the theorem holds for \( n - 1 \) where \( n - 1 \geq 9 \). We need to show that this result holds for \( n \geq 10 \).

Suppose $M$ is matching in the hypercube $Q_n$. Choose a direction $j\in [n]$, with $|M\cap E_j|$ is small as possible. Since $|M|\leq 3n-13$, there exists $j$ with $|M\cap E_j|\leq 2$. Split $Q_n - E_j$ into subcubes $Q_{n-1}^\alpha$ and their respective matchings $M_\alpha$ where $\alpha\in \{0, 1\}$ in these subcubes. Let denote $M\cap E_j = M_c$, when $|M_c| = 1$, we denote $M_c= \{u_0u_1\}$ and $|M_c| = 2$, denote $M_c=\{u_0u_1, v_0v_1\}$. Note that every vertex $s_{\alpha} \in V(Q_{n-1}^{\alpha})$ has a distinct adjacent vertex $s_{1-\alpha}$ within $Q_{n-1}^{1-\alpha}$, where $ \alpha =\{0, 1\}$. By symmetry, assume \( |M_0| \geq |M_1| \) which implies \( |M_1| \leq |M_0| \leq 3n - 13 \).

Now, consider \( x \) and \( y \) where \( p(x) \neq p(y) \) are any two vertices in \( Q_n \). We are to show that a Hamiltonian path \( P_{xy} \) exists in \( Q_n \). To proceed, we examine four distinct cases.

Case 1. $|M_0| \leq 3n - 16$. Now, \( |M_1| \leq |M_0| \leq 3(n-1) - 13 \). $|M_c| \leq 2$.

Subcase 1.1. \( x, y \in V(Q_{n-1}^0) \) or \( (x, y \in V(Q_{n-1}^1)) \)

By induction a Hamiltonian path \( P^0_{xy} \) exists in \( Q_{n-1}^0 \) passing through $M_0$. If \( M_c = \{u_0u_1\} \), choose a neighbor \( v_0 \) of \( u_0 \) on \(  P^0_{xy} \). Since \( M \) is a matching and \( u_0u_1 \in M_c \), it follows that \( \{u_0v_0, u_1v_1\} \cap M = \emptyset \). If \( M_c = \emptyset \), we have $|E( P^0_{xy})| - (|M_0| + |M_1|) \geq 2^{n-1} -1- (3n - 13) > 1$ for  $n \geq 10$, then an edge \( u_0v_0 \in E( P^0_{xy}) \setminus M_0 \) exists with \( u_1v_1 \notin M_1 \). When \( M_c = \{u_0u_1, v_0v_1\} \) and \( d_{ P^0_{xy}}(u_0, v_0) = 1 \), it follows that \( u_0v_0 \in E( P^0_{xy}) \setminus M_0 \) and \( u_1v_1 \notin M_1 \). In both cases, $|M_1| \leq \frac{3n - 13}{2}< 2n-4$ for $n \geq 10$. By Theorem \ref{n1}, a Hamiltonian path \(  P^0_{u_1v_1} \) exists in \( Q^1_{n-1} \) that containing \( M_1  \). The required Hamiltonian path is $ P_{xy}:= P^0_{xy} +  P^0_{u_1v_1} + \{u_0u_1, v_0v_1\} - u_0v_0$, of \( Q_n \) passing through $M$.

If \( M_c = \{u_0u_1, v_0v_1\} \) and \( d_{ P^0_{xy}}(u_0, v_0) > 1 \), let \( s_0 \) and \( r_0 \) be the respective neighbors of \( u_0 \) and \( v_0 \) on \( P^0_{xy} \), such that $s_0 \neq r_0$ and \( s_1r_1 \notin M_1 \). Given that \( \{s_1, r_1\} \cap V(M_1) = \emptyset \) and $|M_1| \leq \frac{3n-13}{2} \leq 2(n-1)- 8$ for $n \geq 10$, it follows that \( \{u_0s_0, v_0r_0\} \in E(P^0_{xy}) \setminus M_0 \) and \( \{u_1s_1,v_1r_1\} \notin M_1 \). Using Lemma \ref{wz}, a spanning 2-path  \( P^1_{u_1s_1}+ P^1_{v_1r_1} \) exists in \( Q^1_{n-1} \) containing \( M_1 \). Thus, the desired Hamiltonian path is  $ P_{xy}:= P^0_{xy} + P^1_{u_1s_1} + P^1_{v_1r_1} + \{u_0u_1, s_0s_1, v_0v_1, r_0r_1\}  - \{u_0s_0, v_0r_0\}$, containing $M$ in \( Q_n \).

Subcase 1.2. \( x \in V(Q_{n-1}^0)\) and \( y \in V(Q_{n-1}^1)  \) or (\( y \in V(Q_{n-1}^0)\) and \( x \in V(Q_{n-1}^1)\)).

Select a vertex \( r_0 \in V(Q_{n-1}^0) \) with \( p(r_0) \neq p(x) \) and $ yr_1, xr_0 \notin M_0$. Moreover, it holds that \( p(y) \neq p(r_1) \) and $r_1\notin V(M_1)$. By induction hypothesis, a Hamiltonian path \( P^0_{xr_0} \) exists in \( Q_{n-1}^0 \) containing $M_1$. If \( M_c = \emptyset \), and we have $|M_1| \leq \frac{3n - 13}{2}< 2(n-1)-4$ for $n \geq 10$. By Theorem \ref{n1}, there exists a Hamiltonian path $P^1_{yr_1}$ in $Q_{n-1}^1$ passing through $M_1$. Thus, the desired Hamiltonian path is  $P_{xy}:= P^0_{xr_0} + P^1_{yr_1} +  r_0r_1$ in \( Q_n \) passing through $M$.

Since $|M_c|\leq 2$. We choose the edges on $P_{xr_0}^0$ according to the following rules.

(i) If \( |M_c| = 1 \), let denote \( M_c = \{u_0u_1\} \), choose a neighbor \( s_0 \) of \( u_0 \). We have $|E( P^0_{xr_0})| - (|M_0| + |M_1|) \geq 2^{n-1} -1- (3n - 13) > 2$ for  $n \geq 10$, then there exists an edge $v_0t_0$ on \(  P^0_{xr_0} \) such that $\{u_0s_0, u_1s_1, v_0t_0, v_1t_1\} \cap M = \emptyset$ satisfying $\{u_1, v_1, r_1\}\cap V(M_1)=\emptyset$, and \( \{s_1y, t_1y, s_1t_1\} \notin M_1 \).

(ii) Otherwise, \( |M_c| = 2 \), let denote $ M_c = \{u_0u_1, v_0v_1\}$ and \( d_{ P^0_{xr_0}}(u_0, v_0) > 1 \). Let \( s_0 \) and \( t_0 \) be the respective neighbors of \( u_0 \) and \( v_0 \) on \( P^0_{xr_0} \), such that $s_0 \neq t_0$ and \( s_1t_1 \notin M_1 \). When \( M_c = 2\), let denote $ M_c = \{u_0u_1, s_0s_1\}$ and \( d_{ P^0_{xy}}(u_0, s_0) = 1 \), it follows that \( u_0s_0 \in E( P^0_{xy}) \setminus M_0 \) and \( u_1s_1 \notin M_1 \). Note that $|E( P^0_{xy})| - (|M_0| + |M_1|) \geq 2^{n-1} -1- (3n - 13) > 2$ for  $n \geq 10$, then there exists an edge $v_0t_0$ on \(  P^0_{xr_0} \) such that $\{u_0s_0, u_1s_1, v_0t_0, v_1t_1\} \cap M = \emptyset$ satisfying $\{u_1, v_1, r_1\}\cap V(M_1)=\emptyset$, and \( \{s_1y, t_1y, s_1t_1\} \notin M_1 \).

Since $|M_1| \leq \frac{3n-14}{2} \leq 2(n-1)- 10$ for $n \geq 10$, it follows that \( \{u_0s_0, v_0t_0\} \in E(P^0_{xr_0}) \setminus M_0 \) and $p(y)\neq p(r_1)$. By Lemma \ref{2.11}, a spanning 3-path \( P^1_{u_1s_1}+ P^1_{v_1t_1} + P_{yr_1}^1 \) exists in \( Q^1_{n-1} \) containing \( M_1 \). Thus, the required Hamiltonian path is $P_{xy}:= P^0_{xr_0} + P^1_{u_1s_1} + P_{v_0t_1}^1 + P^1_{yr_1} + \{u_0u_1, v_0v_1, s_0s_1, r_0r_1, t_0t_1\}  - \{u_0s_0, v_0t_0\}$ passing through $M$ in \( Q_n \).
	
Case 2. $|M_0| = 3n - 15$. Now, \( |M_1| \leq 2 \). 

Subcase 2.1. If $x, y \in V(Q_{n-1}^0)$.

Since $|M_0| = 3n - 15$, choose an edge $e \in M_0$ then $|M_0 \setminus \{e\}| = 3n - 16 = 3(n - 1) - 13$, applying induction hypothesis, a Hamiltonian path $P^0_{xy}$ exists in $Q_{n-1}^0$ passing through $(M_0 \setminus \{e\})$. 

If \( M_c = \emptyset \), then \( |M_1| \leq 2 \). When \( |M_c| \geq 1  \), let $s_0r_0$ be the chosen edge with $d(u_0, s_0r_0)\neq 1$ such that $s_0r_0\notin P_{xy}^0$. Choose \( a_0 \) and \( b_0 \) be the neighbors of \( s_0 \) and \( r_0 \) on \( P^0_{xy}\), respectively, such that the path between \(s_0 \) and \( r_0 \) on \( P^0_{xy} \) passes through \( a_0 \) and \( b_0 \) is on other side. If \( a_1b_1 \notin M_1 \), and since \( p(a_1) \neq p(b_1) \),  by Lemma \ref{a4}, there exists a Hamiltonian path \( P^1_{a_1b_1} \) containing \( M_1 \) in \( Q^1_{n-1} \). The desired Hamiltonian path is $P_{xy}:= P^0_{xy} + P^1_{a_1b_1} + \{a_0a_1, b_0b_1, s_0r_0\} - \{s_0a_0, r_0b_0\}$ in \( Q_n \) containing $M$. If, however, \( a_1b_1 \in M_1 \), then since $|E(P^0_{xy}) \setminus M_0| - |M_1| \geq 2^{n - 1} -1- (3n - 13) > 5$, an edge \( w_0t_0 \in E(P^0_{xy}) \setminus M_0 \) exists with \( \{w_0, t_0\} \cap \{a_0, b_0\} = \emptyset \) and \( w_1t_1 \notin M_1 \). By Lemma \ref{a1} if \( |M_1| = 2 \), or Lemma \ref{dt} if \( |M_1| = 1 \), there is a Hamiltonian path \( P^1_{w_1t_1} \) passing through \( M_1 \setminus \{a_1b_1\} \) in \( Q^1_{n-1} - \{a_1, b_1\} \). The desired Hamiltonian path is $P_{xy}:= P^0_{xy} + P^1_{w_1t_1} + \{s_0r_0, a_0a_1, b_0b_1, w_0w_1, t_0t_1\} - \{s_0a_0, r_0b_0, w_0t_0\}$, in \( Q_n \) passing through $M$, depicted in Figure \ref{Fig.4}(a).
\begin{figure}[H]
	\centering
	\includegraphics[width=0.95\linewidth, height=0.26\textheight]{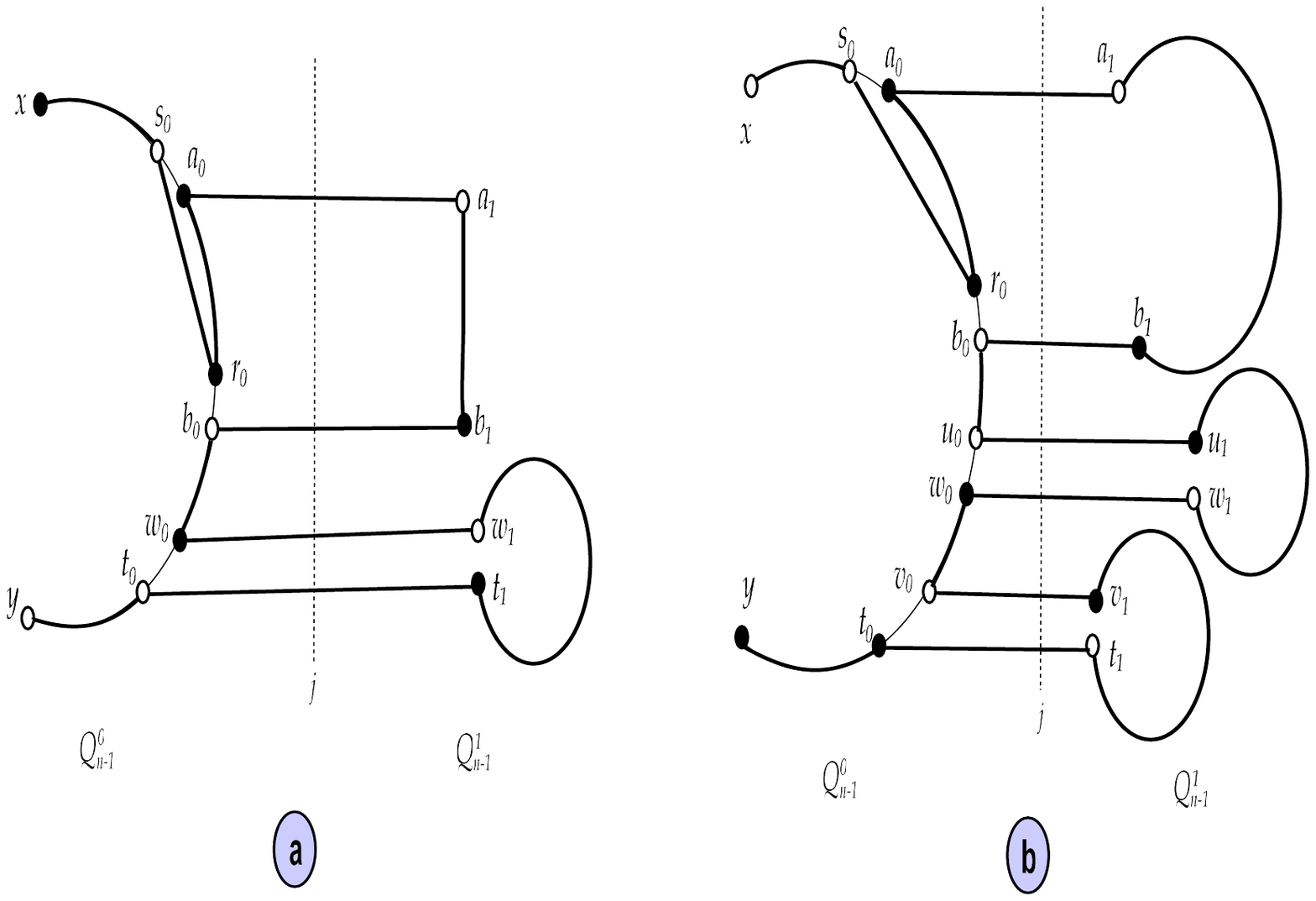}
	\caption[A]{\small Illustration of subcase 2.1.}\label{Fig.4}
\end{figure}

If \( M_c = \{u_0u_1\} \), then \( |M_1| \leq 1 \). Choose a neighbor \( v_0 \) of \( u_0 \) on \( P^0_{xy} \) such that \( v_1 \notin V(M_1) \). This ensures that \( \{u_1, v_1\} \cap V(M_1) = \emptyset \). If \( M_c = \{u_0u_1, v_0v_1\} \) and \( d_{P^0_{xy}}(u_0, v_0) = 1 \), it follows that \( M_1 = \emptyset \). In both of these cases, since \( u_0 \) is not adjacent to \( s_0 \) or \( r_0 \) in \( P^0_{xy} \), we can select neighbors \( a_0 \) and \( b_0 \) of \( s_0 \) and \( r_0 \) on \( P^0_{xy} \), respectively, such that the path between \( s_0 \) and \( r_0 \) on \( P^0_{xy} \) contains \( a_0 \) and \( b_0 \) is on other side. We have \( \{a_0, b_0\} \cap \{u_0, v_0\} = \emptyset \). Since \( p(a_1) \neq p(b_1) \) and \( p(u_1) \neq p(v_1) \) such that $\{u_1, v_1\}\cap M_1= \emptyset$, by applying Lemma \ref{a5} when \( M_1 = 1 \), or Lemma \ref{dt} when \( M_1 = \emptyset \), a spanning 2-path \( P^1_{a_1b_1}+ P^1_{u_1v_1} \) of \( Q_{n-1}^1 \) that passes through \( M_1 \). The desired Hamiltonian path is $P_{xy}:= P^0_{xy} + P^1_{a_1b_1} + P^1_{u_1v_1} + \{s_0r_0, a_0a_1, b_0b_1, u_0u_1, v_0v_1\} - \{s_0a_0, r_0b_0, u_0v_0\}$, in \( Q_n \) passing through $M$.

If \( M_c = \{ u_0u_1, v_0v_1 \} \) and \( d_{P_{xy}^0}(u_0, v_0) > 1 \), then \( M_1 = \emptyset \). Choose neighbors \( a_0 \) and \( b_0 \) of \( s_0 \) and \( r_0 \) on \( P^0_{xy} \), respectively, such that the path between \( s_0 \) and \( r_0 \) passes through \( a_0 \) and \( b_0 \) on other side. Observe that there are two possible ways to select \( p(s_0) \neq p(r_0) \), the neighbors of \( s_0 \) and \( r_0 \) on \( P^0_{xy} \) are disjoint. Therefore, we can select \( a_0 \) and \( b_0 \) such that \( v_0 \notin \{ a_0, b_0 \} \). Since \( u_0 \) is not adjacent to \( s_0 \) or \( r_0 \) on \( P^0_{xy} \), we have \( \{u_0, v_0\} \cap \{a_0, b_0\} = \emptyset \). Next, choose neighbors \( w_0 \) and \( t_0 \) of \( u_0 \) and \( v_0 \) on \( P^0_{xy} \), respectively, such that \( a_0 \neq b_0 \) and \( \{w_0, t_0\} \cap \{a_0, b_0\} = \emptyset \). Since \( d_{P^0_{xy}}(u_0, v_0) > 1 \), this is always possible. Now, \( a_0, b_0, u_0, v_0, w_0, t_0 \) are distinct vertices, and we have \( p(u_1) \neq p(w_1) \), \( p(v_1) \neq p(t_1) \), and \( p(a_1) \neq p(b_1) \). By Theorem \ref{a8}, a spanning 3-path  \(  P^1_{u_1w_1}+P^1_{v_1t_1}+P^1_{a_1b_1} \) exists in \( Q_{n-1}^1 \). Hence, the required Hamiltonian path is $P_{xy}:= P^0_{xy} + P^1_{u_1w_1}+P^1_{v_1t_1}+P^1_{a_1b_1} + \{ s_0r_0, u_0u_1, v_0v_1, a_0a_1, b_0b_1, w_0w_1, t_0t_1 \} - \{ s_0a_0, r_0b_0, u_0w_0, v_0t_0 \}$, in \( Q_n \) passing through $M$, shown in Figure \ref{Fig.4}(b).

Subcase 2.2. $x \in V(Q_{n-1}^0)$ and $y \in V(Q_{n-1}^1)$.

Choose $z_0 \in V(Q_{n-1}^0)$ and $z_0x\notin M_0$ with $p(z_0) \neq p(x)$ there exists $p(z_1) \neq p(y)$, such that the vertices $z_1, y$ are distinct and $z_1y\notin M$. Since $|M_0 \setminus \{s_0r_0\}| = 3(n - 1) - 13$, by induction a Hamiltonian path $P^0_{xz_0}$ exists in $Q_{n-1}^0$ passing through $(M_0 \setminus \{s_0r_0\})$. If \( M_c = \emptyset \), then $|M_1| \leq 2$, and if $s_0r_0 \in E(P^0_{xz_0})$, by Lemma \ref{a4}, a Hamiltonian path $P^1_{z_1y}$ exists in $Q_{n-1}^1$ containing $M_1$. Thus, the required Hamiltonian path is $P_{xy}:= P^0_{xz_0} + z_0z_1 + P^1_{z_1y}$ in $Q_n$ passing through $M$. 

When \( |M_c| \geq 1  \), let $s_0r_0$ be chosen edge with $d(u_0, s_0r_0)\neq 1$ and $s_0r_0\notin E(P_{xz_0}^0)$ where \( a_0 \) and \( b_0 \) be the neighbors of \( s_0 \) and \( r_0 \) on \( P^0_{xz_0}\), respectively, such that the path between \(s_0 \) and \( r_0 \) on \( P^0_{xz_0} \) passing through \( a_0 \) and \( b_0 \) is on other side, such that \( a_1b_1 \notin M_1 \). Since \( p(a_1) \neq p(b_1) \), and \( \{a_1, z_1\} \cap M_1 = \emptyset \)  by Lemma \ref{a6}, there is a spanning 2-path \( P^1_{a_1b_1} + P_{yz_1}^1\) exists passing through \( M_1 \) in \( Q^1_{n-1} \). The desired Hamiltonian path is $P_{xy}:= P^0_{xz_0} + P^1_{a_1b_1} + P_{yz_1}^1 + \{a_0a_1, b_0b_1, s_0r_0, z_0z_1\} - \{s_0a_0, r_0b_0\}$, in \( Q_n \)containing $M$. If, however, \( a_1b_1 \in M_1 \), by Lemma \ref{a1} if \( |M_1| = 2 \), or Lemma \ref{dt} if \( |M_1| = 1 \), there is a Hamiltonian path \( P^1_{yz_1} \) passing through \( M_1 \setminus \{a_1b_1\} \) in \( Q^1_{n-1} - \{a_1, b_1\} \). The desired Hamiltonian path is $P_{xy}:=P^0_{xz_0} + P^1_{yz_1}+\{s_0r_0, a_0a_1, b_0b_1, z_0z_1\} - \{s_0a_0, r_0b_0\}$ in \( Q_n \) passing through $M$, illustrated in Figure \ref{Fig.3}(a).
\begin{figure}[H]
	\centering
	\includegraphics[width=0.95\linewidth, height=0.26\textheight]{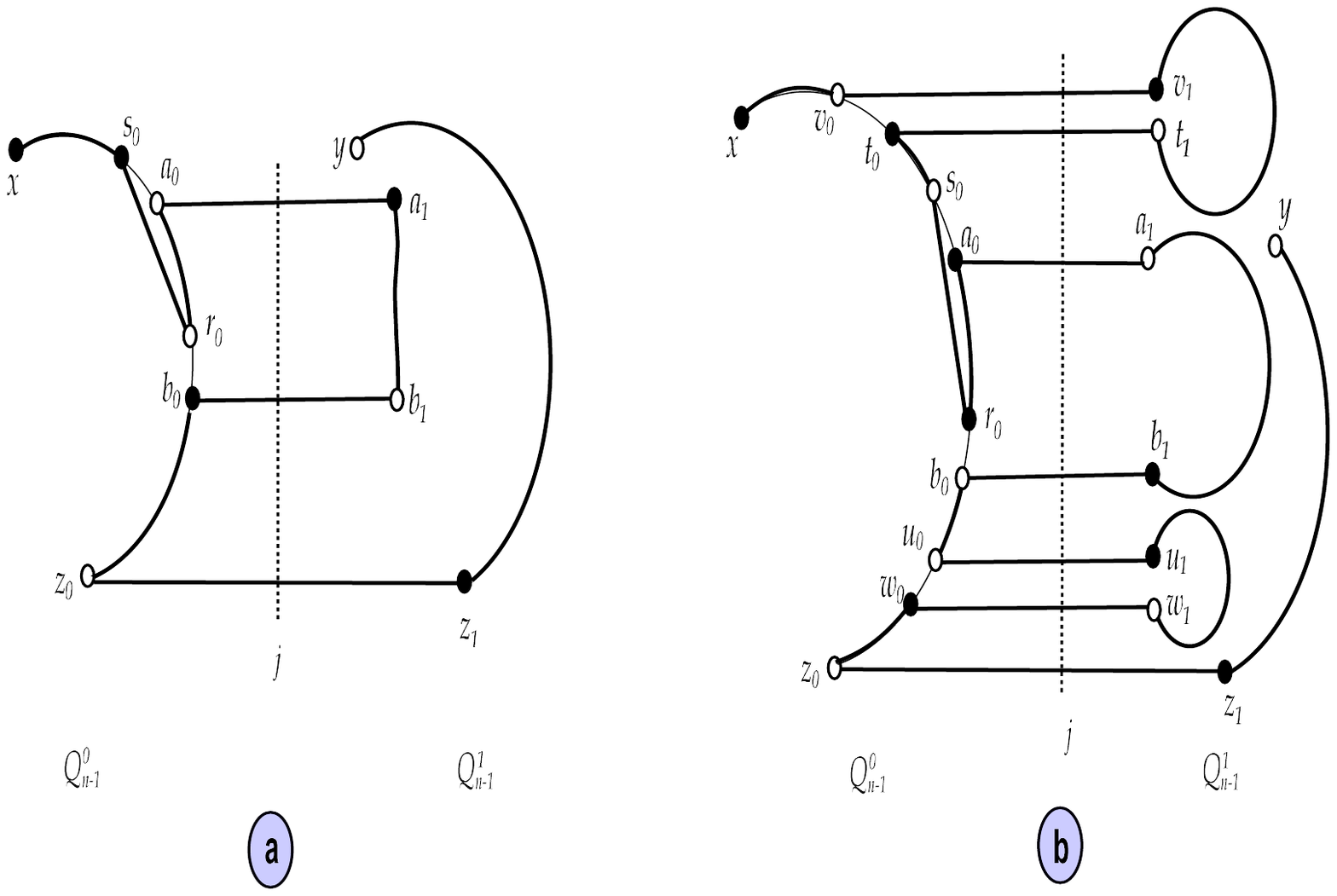}
	\caption[A]{Illustration of subcase 2.2.}\label{Fig.3}
\end{figure}

If \( M_c = \{u_0u_1\} \), then \( |M_1| \leq 1 \) and  $a_1b_1 \notin M_1$. Select a neighbor \( v_0 \) of \( u_0 \) in \( P^0_{xz_0} \) with \( v_1 \notin V(M_1) \). This ensures that \( \{u_1, v_1\} \cap V(M_1) = \emptyset \). If \( M_c = \{u_0u_1, v_0v_1\} \) and \( d_{ P^0_{xz_0}}(u_0, v_0) = 1 \), it follows that \( M_1 = \emptyset \). In both of these cases, since \( u_0 \) is not adjacent to \( s_0 \) or \( r_0 \) in \( P^0_{xy} \), we can select neighbors \( a_0 \) and \( b_0 \) of \( s_0 \) and \( r_0 \) on \( P^0_{xz_0} \), respectively, such that the path between \( s_0 \) and \( r_0 \) on \( P^0_{xz_0} \) contains \( a_0 \) and \( b_0 \) is on other side, such that \( \{a_0, b_0\} \cap \{u_0, v_0\} = \emptyset \). Since \( p(a_1) \neq p(b_1) \), \( p(u_1) \neq p(v_1) \) and \( \{a_1, b_1, u_1, v_1\} \cap M_1 = \emptyset \), by applying Lemma \ref{a7} when \( M_1 = 1 \), or Lemma \ref{a8} when \( M_1 = \emptyset \), a spanning 3-path \( P^1_{a_1b_1}+ P^1_{u_1v_1} + P^1_{yz_1} \) of \( Q_{n-1}^1 \) passes through \( M_1 \). The desired Hamiltonian path is  $P_{xy}:= P^0_{xz_0} +P^1_{a_1b_1}+ P^1_{u_1v_1} + P^1_{yz_1} + \{s_0r_0, a_0a_1, b_0b_1, u_0u_1, v_0v_1, z_0z_1\} - \{s_0a_0, r_0b_0, u_0v_0\}$ in \( Q_n \) passing through $M$. If  $a_1b_1 \in M_1$ we have \( \{z_1, y, u_1, v_1\} \cap M_1 = \emptyset \), Lemma \ref{a7} a spanning 3-path \( P^1_{u_1v_1} + P^1_{yz_1} + a_1b_1 \) of \( Q_{n-1}^1 \) containing \( M_1 \). The required Hamiltonian path is  $P_{xy}:= P^0_{xz_0} + P^1_{u_1v_1} + P^1_{yz_1} + \{a_1b_1, s_0r_0, a_0a_1, b_0b_1, u_0u_1, v_0v_1, z_0z_1\} - \{s_0a_0, r_0b_0, u_0v_0\}$ in \( Q_n \) containing $M$.

If \( M_c = \{ u_0u_1, v_0v_1 \} \) and \( d_{ P^0_{xz_0}}(u_0, v_0) > 1 \), then \( M_1 = \emptyset \). Choose neighbors \( a_0 \) and \( b_0 \) of \( s_0 \) and \( r_0 \) on \( P^0_{xz_0} \), respectively, such that the path joining \( s_0 \) and \( r_0 \) passes through \( a_0 \) and  \( b_0 \) is on other side. Observe that there are two possible ways to select \( a_0 \) and \( b_0 \). Since \( Q_n \) is bipartite and \( p(s_0) \neq p(r_0) \), the neighbors of \( s_0 \) and \( r_0 \) on \( P^0_{xz_0} \) are disjoint. Therefore, we can select \( a_0 \) and \( b_0 \) such that \( v_0 \notin \{ a_0, b_0 \} \). Since \( u_0 \) is not adjacent to \( s_0 \) or \( r_0 \) on \( P^0_{xz_0} \), we have \( \{u_0, v_0\} \cap \{a_0, b_0\} = \emptyset \). Next, choose neighbors \( w_0 \) and \( t_0 \) of \( u_0 \) and \( v_0 \) on \( P^0_{xz_0} \), respectively, such that \( w_0 \neq t_0 \) and \( \{w_0, t_0\} \cap \{a_0, b_0\} = \emptyset \). Since \( d_{P^0_{xz_0}}(u_0, v_0) > 1 \), this is always possible. Now, \( a_1, b_1, u_1, v_1, w_1, t_1, y, z_1 \) are distinct vertices, and we have \( p(u_1) \neq p(w_1) \), \( p(v_1) \neq p(t_1) \), $p(z_1)\neq p(y)$ and \( p(a_1) \neq p(b_1) \). By Theorem \ref{a8}, a spanning 4-path  \(  P^1_{u_1w_1}+P^1_{v_1t_1}+P^1_{a_1b_1} + P^1_{yz_1} \) exists in \( Q_{n-1}^1 \). Hence, the required Hamiltonian path in \( Q_n \) is $P_{xy}:= P^0_{xz_0} + P^1_{u_1w_1} + P^1_{v_1t_1} + P^1_{a_1b_1} + P^1_{yz_1} + \{s_0r_0, u_0u_1, v_0v_1, a_0a_1, b_0b_1, w_0w_1, t_0t_1, z_0z_1 \} - \{s_0a_0, r_0b_0, u_0w_0, v_0t_0 \}$,  depicted in Figure \ref{Fig.3}(b).

Subcase 2.3. \( x, y \in V(Q_{n-1}^1) \).  

Since \( |M_0| = 3n - 15 < 3n -10 \) then by Theorem \ref{wz16}, a Hamiltonian cycle $C_0$ exists passing through $M_0$ in $Q_{n-1}^0$. If \( M_c = \emptyset \), then \( |M_1| \leq 2 \). Choose an $s_0r_0 \notin M_0$ and $s_1r_1 \notin M_1$ on $C_0$ such that \( M_1 \cup \{s_1r_1\} \) form a linear forest. We have \( p(x) \neq p(y) \) and $xy \neq M_1$,  by Theorem \ref{n1}, there exists a Hamiltonian path \( P^1_{xy}\) containing \( M_1 \cup \{s_1r_1\}\) in \( Q^1_{n-1} \). The required Hamiltonian path is $P_{xy}:= C_0 + P^1_{xy} + \{s_0s_1, r_0r_1\} - \{s_0r_0, s_1r_1\}$ in \( Q_n \) passing through $M$.
 
If \( M_c = \{u_0u_1\} \), then \( |M_1| \leq 1 \). Choose a neighbor \( v_0 \) of \( u_0 \) on \( C_0 \) with \( v_1 \notin V(M_1) \). This ensures that \( \{u_1, v_1\} \cap V(M_1) = \emptyset \). If \( M_c = \{u_0u_1, v_0v_1\} \) and \( d_{C_0}(u_0, v_0) = 1 \), it follows that \( M_1 = \emptyset \). Since \( p(u_1) \neq p(x) \) and \( p(v_1) \neq p(y) \), by applying Lemma \ref{a5} when \( M_1 = 1 \), or Lemma \ref{dt} when \( M_1 = \emptyset \), a spanning 2-path \( P^1_{u_1x}+ P^1_{v_1y} \) of \( Q_{n-1}^1 \) that passes through \( M_1 \). The desired Hamiltonian path is  $P_{xy}:= C_0 + P^1_{u_1x}+ P^1_{v_1y} + \{u_0u_1, v_0v_1\} - u_0v_0$ in \( Q_n \) containing $M$.
 
If \( M_c = \{ u_0u_1, v_0v_1 \} \) and \( d_{C_0}(u_0, v_0) > 1 \), then \( M_1 = \emptyset \). Choose neighbors of $w_0$, $t_0$ of $u_0, v_0$, respectively on $C_0$ such $\{u_0w_0, v_0t_0\}\notin M_0$.  Since \( u_1, v_1, w_1, t_1, x, y \) are distinct vertices, and we have \( p(u_1) \neq p(x) \), \( p(w_1) \neq p(y) \), and \( p(v_1) \neq p(t_1) \). By Theorem \ref{a8}, there exists a spanning 3-path  \(  P^1_{u_1x}+P^1_{w_1y}+P^1_{v_1t_1} \) of \( Q_{n-1}^1 \). Hence, the required Hamiltonian path is $P_{xy}:= C_0 + P^1_{u_1x}+P^1_{w_1y} + P^1_{v_1t_1} + \{u_0u_1, v_0v_1, w_0w_1, t_0t_1\} - \{u_0w_0, v_0t_0 \}$ in \( Q_n \) containing $M$.

Case 3: $|M_0| = 3n - 14$. Now, $M_1 \leq 1$.

Subcase 3.1. If $x, y \in V(Q_{n-1}^0)$.

Since $|M_0| = 3n - 14$, there exist two edges $e, f \in M_0$ then $|M_0 \setminus \{e, f\}| = 3n - 16 = 3(n - 1) - 13$, applying induction hypothesis, a Hamiltonian path $P^0_{xy}$ exists in $Q^0_{n-1}$ passing through $(M_0 \setminus \{e, f\})$.

When \( |M_c| \geq 1  \), let $s_0r_0, w_0t_0$ be chosen edges such that $d(u_0, s_0r_0)\neq 1$ and  $d(u_0, w_0t_0)\neq 1$. Since \(\{s_0r_0, w_0t_0\} \cap E(P_{xy}^0) = \emptyset\), there are two distinct possibilities up to isomorphism. Let \(a_0, b_0, c_0, d_0\) represent the neighbors of \(s_0, r_0, w_0, t_0\) on \(P_{xy}^0\), respectively. Note that \( a_0, b_0, c_0, d_0 \) are not arbitrary neighbors but are specifically chosen based on their positions on the path. It is also possible that \(a_0 = w_0\) or \(d_0 = r_0\), see in Figure \ref{Fig.2}(a), or even that \(a_0 = w_0\), \(b_0 = t_0\), or \(c_0 = r_0\),  see in Figure \ref{Fig.2}(b). While the following constructions remain valid in all cases.
\begin{figure}[H]
	\centering
	\includegraphics[width=0.92\linewidth, height=0.3\textheight]{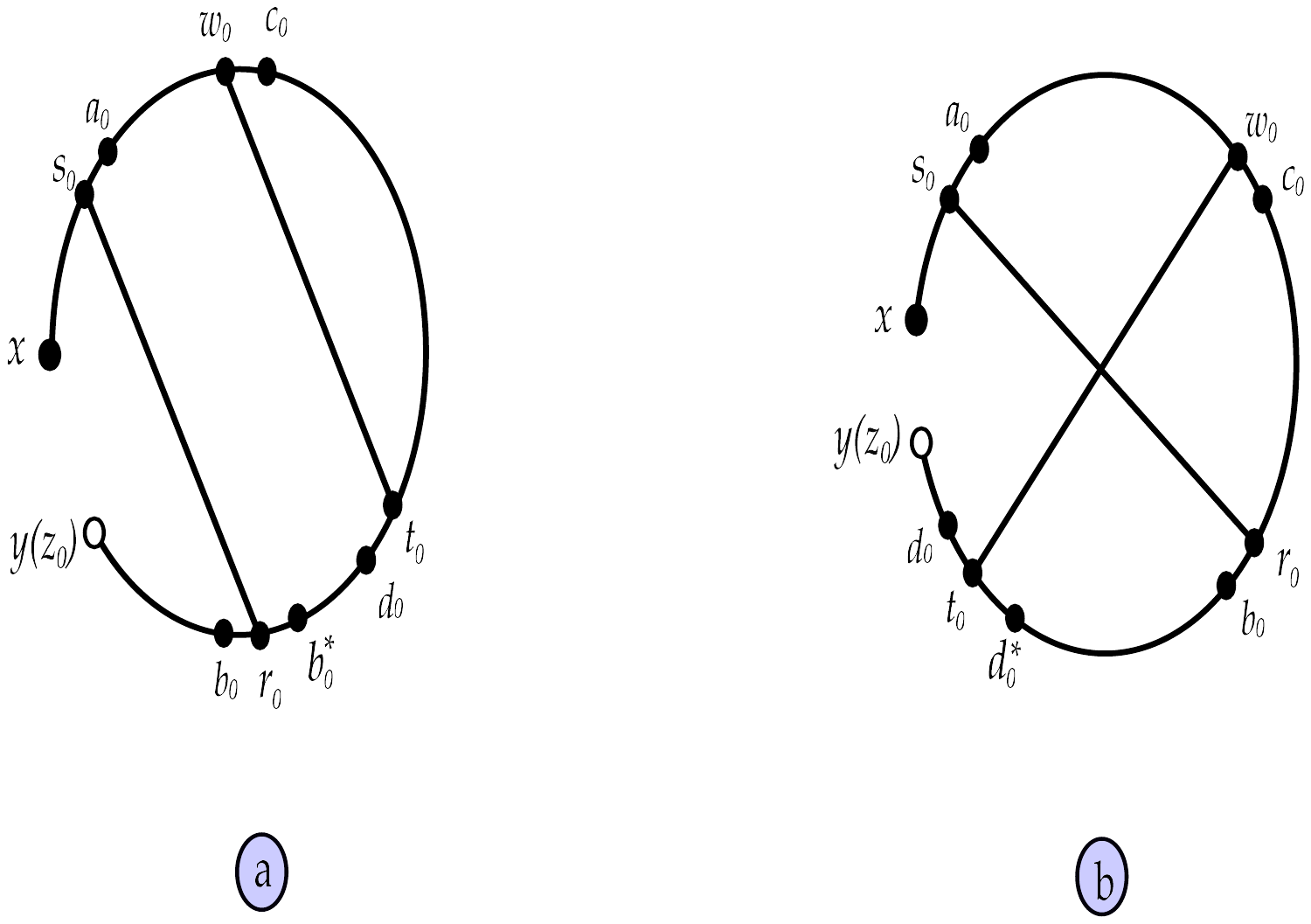}
	\caption[A]{\small Two possibilities of edges $\{s_0r_0, w_0t_0\}$ and its neighbors.}\label{Fig.2}
\end{figure}

If \(M_c = \{u_0u_1\}\), then \(M_1 = \emptyset\). Note that \(a_0\), \(b_0\), \(c_0\), and \(d_0\) are distinct vertices. Since \(u_0\) is not adjacent to any of \(\{s_0, r_0, w_0, t_0\}\) on \(P_{xy}^0\), it follows that \(u_0 \notin \{a_0, b_0, c_0, d_0\}\). Let \(v_0\) be a neighbor of \(u_0\) on \(P_{xy}^0\), such that \(v_0 \notin \{a_0, b_0, c_0, d_0\}\). If \(\{s_0r_0, w_0t_0\}\) is as on Figure \ref{Fig.2}(a), define \(k_1, k_2, k_3, k_4\) as vertices of \(a_1, b_1, c_1, d_1\), respectively. If \(\{s_0r_0, w_0t_0\}\) is as on Figure \ref{Fig.2}(b), let \(k_1, k_2, k_3, k_4\) be vertices of \(a_1, c_1, b_1, d_1\), respectively. Since \(p(a_1) \neq p(b_1)\) and \(p(c_1) \neq p(d_1)\), we have $\{\{k_1, k_2\}, \{k_3, k_4\}, \{u_1, v_1\}\}$ is a balanced set of vertices.  Given that \(n - 1 > 9\), by Theorem \ref{a8}, a spanning 3-path \( P^1_{k_1k_2} + P^1_{k_3k_4} + P^1_{u_1v_1} \) exists in \(Q^1_{n-1}\). Thus, the required Hamiltonian path is $P_{xy}:= P_{xy}^0 +  P^1_{k_1k_2} + P^1_{k_3k_4} + P^1_{u_1v_1} + \{s_0r_0, w_0t_0, a_0a_1, b_0b_1, c_0c_1, d_0d_1, u_0u_1, v_0v_1\} - \{s_0a_0, r_0b_0, w_0c_0, t_0d_0, u_0v_0\}$ in \(Q_n\) containing $M$.

It remains to address the case where \(M_c = \emptyset\), consider two subcases. We have \(|M_1| \leq 1\).

Subcase 3.1.1. Assume \(\{s_0r_0, w_0t_0\}\) is as shown in Figure \ref{Fig.2}(a).

If either \(\{a_1, b_1\} \cap V(M_1) = \emptyset\) or \(\{c_1, d_1\} \cap V(M_1) = \emptyset\), then by applying Lemma \ref{a5} when \(|M_1| = 1\) or Lemma \ref{a8} when \(M_1 = \emptyset \), it is possible to obtain a spanning 2-path \( P^1_{a_1b_1}+ P^1_{c_1d_1} \) of \(Q^1_{n-1}\) containing \(M_1\). Hence, the desired Hamiltonian path is  $P_{xy}:= P_{xy}^0 + P^1_{a_1b_1} + P^1_{c_1d_1} + \{s_0r_0, w_0t_0, a_0a_1, b_0b_1, c_0c_1, d_0d_1\} - \{s_0a_0, r_0b_0, w_0c_0, t_0d_0\}$ in \(Q_n\) containing $M$. Otherwise, if \(|\{a_1, b_1\} \cap V(M_1)| = |\{c_1, d_1\} \cap V(M_1)| = 1\), note that in this case, \(|M_1| = 1\). Let \(M_1 = \{e_1\}\).

If \(a_1d_1 = e_1\) or \(b_1c_1 = e_1\), then using the facts \(p(a_1) \neq p(b_1)\) and \(p(c_1) \neq p(d_1)\), we deduce that \(p(a_1) \neq p(d_1)\) and \(p(b_1) \neq p(c_1)\). Given that either \(\{b_1, d_1\} \cap V(e_1) = \emptyset\) or \(\{a_1, c_1\} \cap V(e_1) = \emptyset\), by Lemma \ref{a5}, a spanning 2-path \( P^1_{a_1d_1} + P^1_{b_1c_1} \) exists in \(Q^1_{n-1}\) passes through $e_1$. Thus, the Hamiltonian path is $P_{xy}:= P_{xy}^0 + P^1_{a_1d_1} + P^1_{b_1c_1} +  \{s_0r_0, w_0t_0, a_0a_1, b_0b_1, c_0c_1, d_0d_1\} - \{s_0a_0, r_0b_0, w_0c_0, t_0d_0\}$ of \(Q_n\) containing $M$,  illustrated in Figure \ref{Fig.5}(a). If \(a_1c_1 = e_1\), let \(b_0^*\) denote the neighbor of \(r_0\) on \(P_{xy}^0\) that is distinct from \(b_0\). Since \(Q_n\) is bipartite and \(p(s_0) \neq p(r_0)\), it follows that \(p(b_0^*) \neq p(d_0)\). Moreover, because \(p(c_1) = p(b^*_1) \neq p(a_1) = p(d_1)\), by Theorem \ref{dt}, there exists a spanning 2-path \(P^1_{b^*_1d_1} + a_1c_1 \) in \(Q^1_{n-1} \). The Hamiltonian path is $P_{xy}:= P_{xy}^0 + P^1_{b^*_1d_1} + \{s_0r_0, w_0t_0, a_1c_1, a_0a_1, b^*_0b^*_1, c_0c_1, d_0d_1\} - \{s_0a_0, r_0b^*_0, w_0c_0, t_0d_0\}$ in \(Q_n\) containing $M$. Eventually, the case \(b_1d_1 = e_1\) is structurally equivalent to the case \(a_1c_1 = e_1\).
\begin{figure}[H]
	\centering
	\includegraphics[width=0.95\linewidth, height=0.25\textheight]{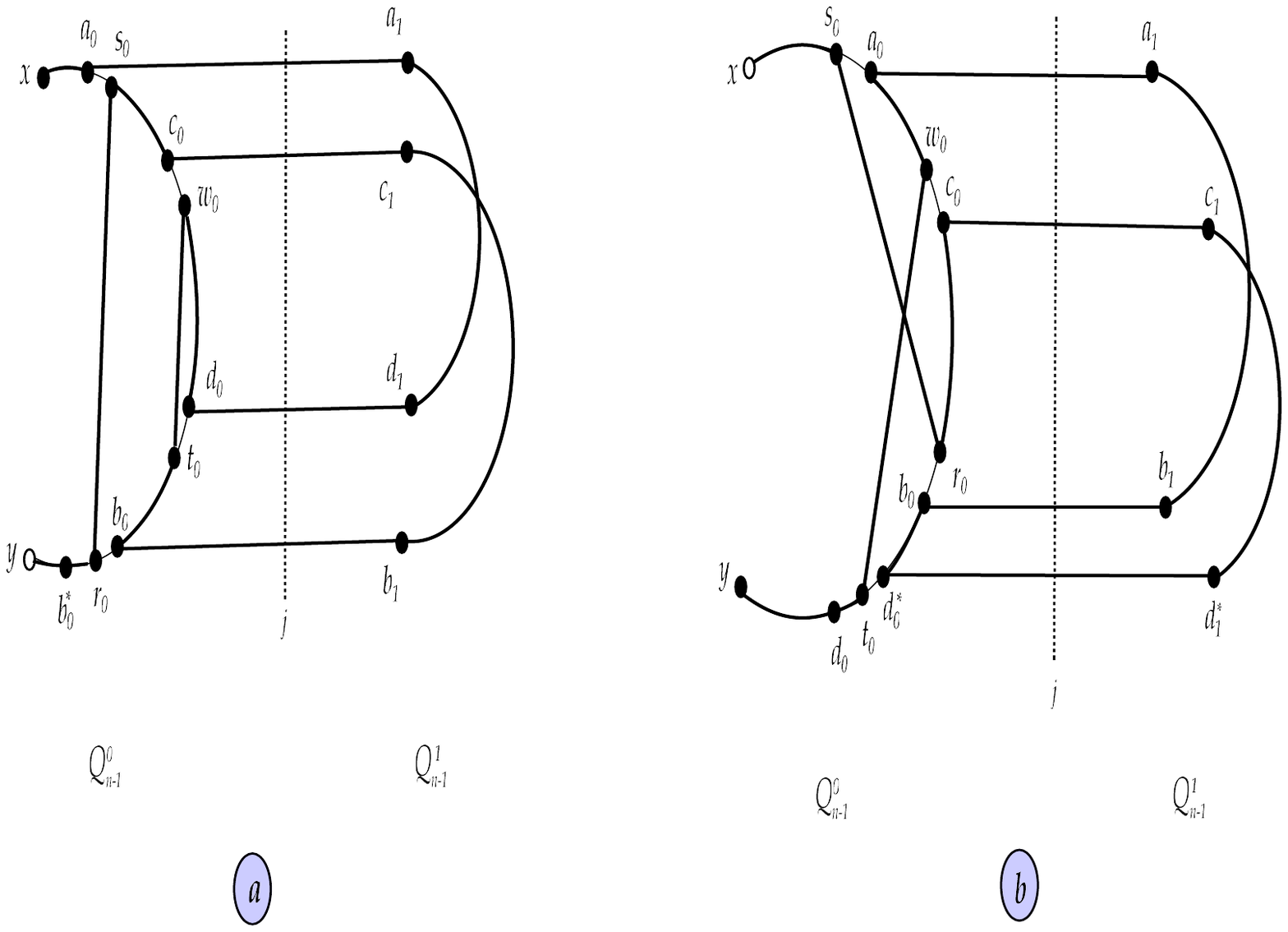}
	\caption[A]{\small Depiction of subcase 3.1.} \label{Fig.5}
\end{figure}

Subcase 3.1.2. Assume that \(\{s_0r_0, w_0t_0\}\) is as shown in Figure \ref{Fig.2}(b). 

Since \(p(w_0) \neq p(t_0)\), it follows that \(p(s_0) = p(w_0)\) or \(p(s_0) = p(t_0)\). W.l.o.g., let \(p(s_0) = p(t_0)\). Thus, we have \(p(s_0) = p(t_0) \neq p(w_0) = p(r_0)\), and \(p(a_1) = p(d_1) \neq p(c_1) = p(b_1)\). If \(\{a_1, c_1\} \cap V(M_1) = \emptyset\) or \(\{b_1, d_1\} \cap V(M_1) = \emptyset\), then, using Lemma \ref{a5} when \(|M_1| = 1\) or Theorem \ref{a8} when \( M_1 = \emptyset \), there exists a spanning 2-path \( P^1_{a_1c_1} + P^1_{b_1d_1} \) of \(Q^1_{n-1}\) that passes through \(M_1\). The desired Hamiltonian path is $P_{xy}:= P_{xy}^0 + P^1_{a_1c_1} + P^1_{b_1d_1} + \{s_0r_0, w_0t_0, a_0a_1, b_0b_1, c_0c_1, d_0d_1\} - \{s_0a_0, r_0b_0, w_0c_0, t_0d_0\}$ in \(Q_n\).

Otherwise, \(|\{a_1, c_1\} \cap V(M_1)| = |\{b_1, d_1\} \cap V(M_1)| = 1\), implying \(|M_1| = 1\).  
Let \(M_1 = \{e_1\}\). Since \(p(a_1) = p(d_1) \neq p(c_1) = p(b_1)\), it follows that \(a_1b_1 = e_1\) or \(c_1d_1 = e_1\). If \(d_{P_{xy}^0}(t_0, r_0) = 1\), then because \(p(a_1) \neq p(c_1)\) and \(a_1c_1 \neq e_1\), by Lemma \ref{a2}, a Hamiltonian path \(P^1_{a_1c_1}\) exists in \(Q^1_{n-1}\) containing $e_1$. The Hamiltonian path is \( P_{xy}:= P_{xy}^0  + P^1_{a_1c_1} + \{s_0r_0, w_0t_0, a_0a_1, c_0c_1\} - \{s_0a_0, w_0c_0, r_0t_0\}\) in \(Q_n\). If \(d_{P_{xy}^0}(t_0, r_0) \neq 1\), let \(d_0^*\) be the neighbor of \(t_0\) on \(P_{xy}^0\) that is distinct from \(d_0\). Since \(p(t_0) \neq p(r_0)\) by assumption, we have \(p(d_0^*) \neq p(c_0)\), which implies \(d_0^* \neq b_0\). Note that \(a_0, b_0, c_0, d_0,\) and \(d_0^*\) are distinct. When \(a_1b_1 = e_1\), it follows that \(\{c_1, d^*_1\} \cap V(e_1) = \emptyset\). When \(c_1d_1 = e_1\), we have \(\{a_1, b_1\} \cap V(e_1) = \emptyset\). Since \(p(d^*_1) = p(t_0) \neq p(w_0) = p(c_1)\), Lemma \ref{a5}, a spanning 2-path \( P^1_{a_1b_1} + P^1_{c_1d^*_1} \) exists in \(Q^1_{n-1}\). The desired Hamiltonian path is \( P_{xy}:= P_{xy}^0 + P^1_{a_1b_1} + P^1_{c_1d^*_1} + \{s_0r_0, w_0t_0, a_0a_1, b_0b_1, c_0c_1, d_0^*d^*_1\} - \{s_0a_0, r_0b_0, w_0c_0, t_0d^*\}\) in \(Q_n\) passing through $M$, illustrated in Figure \ref{Fig.5}(b).

Subcase 3.2. $x \in V(Q_{n-1}^0)$ and $y \in V(Q_{n-1}^1)$.

Choose $z_0 \in V(Q_{n-1}^0)$ with $p(x) \neq p(z_0)$ and $xz_0 \notin M_0$. Moreover, $p(y) \neq p(z_1)$, and the vertices $z_1, y$ are distinct such that $yz_1 \notin M_1$. Since $|M_0 \setminus \{e, f\}| = 3(n - 1) - 13$, by induction a Hamiltonian path $P^0_{xz_0}$  exists in $Q_{n-1}^0$ passing through $(M_0 \setminus \{e, f\})$. Since $|M_1| \leq 1$, if $\{u_0v_0,  w_0z_0\} \in E(P^0_{xz_0})$, then according to Lemma \ref{a2}, a Hamiltonian path $P^1_{z_1y}$ exists in $Q_{n-1}^1$. Thus,  $P_{xy} := P^0_{xz_0} + z_0z_1 + P^1_{z_1y}$ is required Hamiltonian path in $Q_n$ passing through $M$.

When \( |M_c| \geq 1  \), let $s_0r_0, w_0t_0$ be chosen edges such that $d(u_0, s_0r_0)\neq 1$ and  $d(u_0, w_0t_0)\neq 1$. Since \(\{s_0r_0, w_0t_0\} \cap E(P_{xz_0}^0) = \emptyset\), there are two distinct possibilities up to isomorphism, shown in Figure \ref{Fig.2}. Let \(a_0, b_0, c_0, d_0\) represent the neighbors of \(s_0, r_0, w_0, t_0\) on \(P_{xz_0}^0\), respectively. Note that \( a_0, b_0, c_0, d_0 \) are not arbitrary neighbors but are specifically chosen based on their positions on the path. It is also possible that \(a_0 = w_0\) or \(d_0 = r_0\), see in Figure \ref{Fig.2}(a), or even that \(a_0 = w_0\), \(b_0 = t_0\), or \(c_0 = r_0\), see in Figure \ref{Fig.2}(b). While the following constructions remain valid in all cases.

If \(M_c = \{u_0u_1\}\), then \(M_1 = \emptyset\). Note that \(a_0\), \(b_0\), \(c_0\), and \(d_0\) are distinct vertices. Since \(u_0\notin \{s_0, r_0, w_0, t_0\}\) on \(P_{xz_0}^0\), it follows that \(u_0 \notin \{a_0, b_0, c_0, d_0\}\). Let \(v_0\) be a neighbor of \(u_0\) on \(P_{xz_0}^0\), such that \( v_0 \notin \{a_0, b_0, c_0, d_0\}\). If \(\{s_0r_0, w_0t_0\}\) is as on Figure \ref{Fig.2}(a), define \(k_1, k_2, k_3, k_4\) as vertices of \(a_1, b_1, c_1, d_1\), respectively. If \(\{s_0r_0, w_0t_0\}\) is as on Figure \ref{Fig.2}(b), let \(k_1, k_2, k_3, k_4\) be vertices of \(a_1, c_1, b_1, d_1\), respectively. Since \(p(a_1) \neq p(b_1)\) and \(p(c_1) \neq p(d_1)\), we have a balanced set of vertices $\{\{k_1, k_2\}, \{k_3, k_4\}, \{u_1, v_1\}, \{y, z_1\}\}$.  Given that \(n - 1 > 9\), by Theorem \ref{a8}, a spanning 4-path \( P^1_{k_1k_2} + P^1_{k_3k_4} + P^1_{u_1v_1} + P^1_{yz_1} \) exists in \(Q^1_{n-1}\). Therefore, the required Hamiltonian path is $P_{xy}:= P_{xz_0}^0 +  P^1_{k_1k_2} + P^1_{k_3k_4} + P^1_{u_1v_1} + P^1_{yz_1} + \{s_0r_0, w_0t_0, a_0a_1, b_0b_1, c_0c_1, d_0d_1, u_0u_1, v_0v_1, z_0z_1\} - \{s_0a_0, r_0b_0, w_0c_0, t_0d_0, u_0v_0\}$ in \(Q_n\) containing $M$. 

It remains to address the case where \(M_c = \emptyset\), consider two subcases. We have \(|M_1| \leq 1\).

Subcase 3.2.1. Assume \(\{s_0r_0, w_0t_0\}\) is as shown in Figure \ref{Fig.2}(a).

If either \(\{a_1, b_1, c_1, d_1\} \cap V(M_1) = \emptyset\), then by applying Lemma \ref{a7} when \( |M_1| = 1\) or Lemma \ref{a8} when \( M_1 = \emptyset \), it is possible to obtain a spanning 3-path \( P^1_{a_1b_1}+ P^1_{c_1d_1} + P^1_{yz_1} \) of \(Q^1_{n-1}\) containing \(M_1\). Hence, the desired Hamiltonian path $P_{xy}:= P_{xz_0}^0 + P^1_{a_1b_1} + P^1_{c_1d_1} + P^1_{yz_1} + \{s_0r_0, w_0t_0, a_0a_1, b_0b_1, c_0c_1, d_0d_1, z_0z_1\} - \{s_0a_0, r_0b_0, w_0c_0, t_0d_0\}$ in \(Q_n\). Otherwise, if \(|\{a_1, b_1, c_1, d_1\} \cap V(M_1)| = 1\), note that in this case, \(|M_1| = 1\). Let \(M_1 = \{e_1\}\).

If \(a_1d_1 = e_1\) or \(b_1c_1 = e_1\), then using the facts \(p(a_1) \neq p(b_1)\) and \(p(c_1) \neq p(d_1)\), we deduce that \(p(a_1) \neq p(d_1)\) and \(p(b_1) \neq p(c_1)\). Given that either \(\{b_1, d_1, z_1, y\} \cap V(e_1) = \emptyset\) or \(\{a_1, c_1, z_1, y\} \cap V(e_1) = \emptyset\),  by Lemma \ref{a7}, a spanning 3-path \( P^1_{a_1d_1} + P^1_{b_1c_1} +P^1_{yz_1} \) exists in \(Q^1_{n-1}\) passes through $e_1$. Thus, the Hamiltonian path is $P_{xy}:= P_{xz_0}^0 + P^1_{a_1d_1}+P^1_{b_1c_1} +P^1_{yz_1} +  \{s_0r_0, w_0t_0, a_0a_1, b_0b_1, c_0c_1, d_0d_1, z_0z_1\} - \{s_0a_0, r_0b_0, w_0c_0, t_0d_0\}$ of \(Q_n\), see in Figure \ref{Fig.6}(a). If \(a_1c_1 = e_1\), let \(b_0^*\) denote the neighbor of \(r_0\) on \(P_{xz_0}^0\) that is distinct from \(b_0\). Since \(Q_n\) is bipartite and \(p(s_0) \neq p(r_0)\), it follows that \(p(b_0^*) \neq p(c_0)\). Furthermore, because \(p(c_1) = p(b^*_1) \neq p(a_1) = p(d_1)\), by Lemma \ref{a7}, there exists a spanning 3-path \(P^1_{b^*_1d_1}+ P^1_{yz_1} + a_1c_1 \) in \(Q^1_{n-1} \). Hence, the Hamiltonian path is $P_{xz_0}^0 + P^1_{b^*_1d_1} +P^1_{yz_1} + \{s_0r_0, w_0t_0, a_1c_1, a_0a_1, b^*_0b^*_1, c_0c_1, d_0d_1, z_0z_1\} - \{s_0a_0, r_0b^*_0, w_0c_0, t_0d_0\}$ in \(Q_n\) containing $M$. Eventually, the case \(b_1d_1 = e_1\) isomorphic to the case \(a_1c_1 = e_1\).
\begin{figure}[H]
	\centering
	\includegraphics[width=0.92\linewidth, height=0.25\textheight]{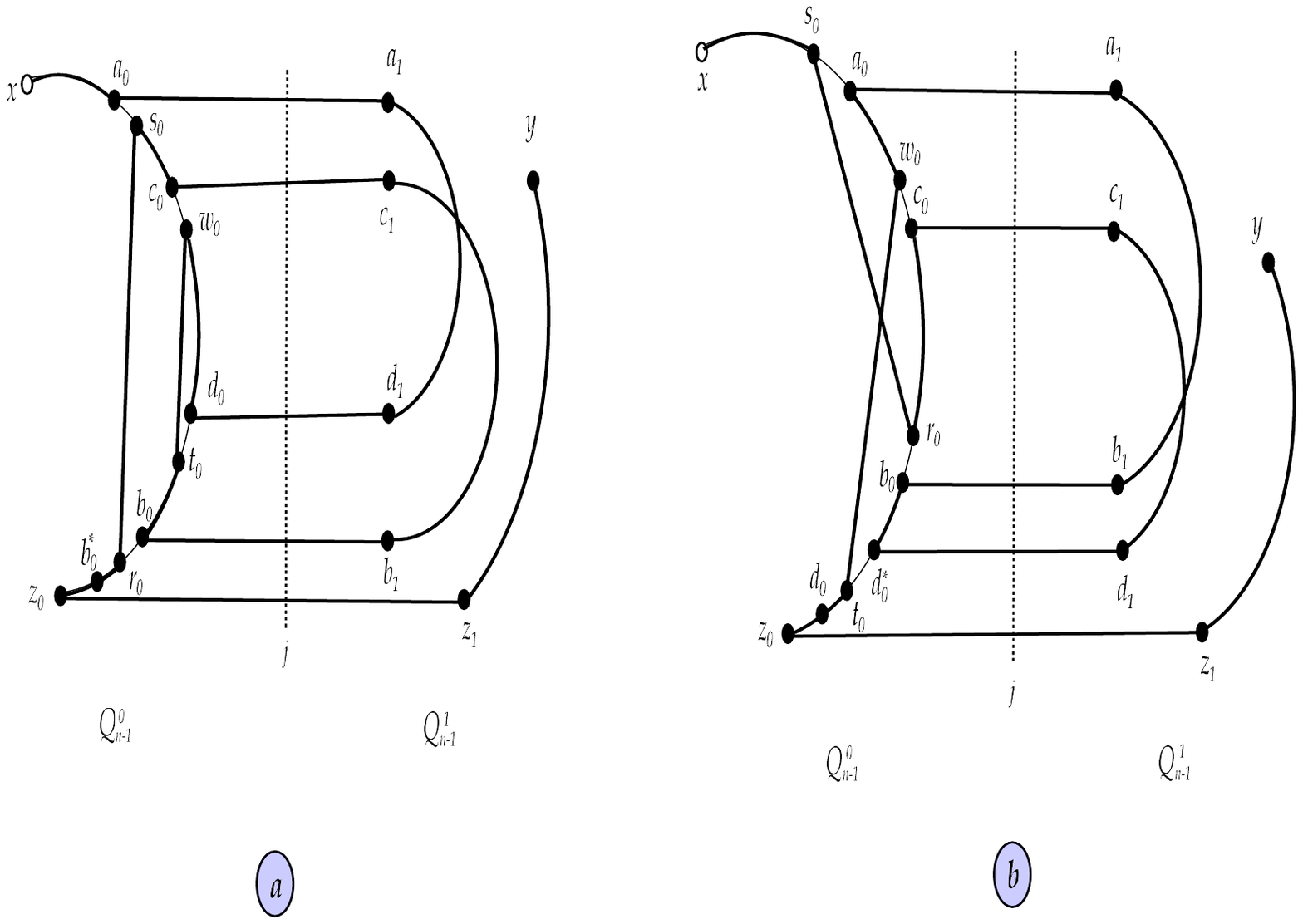}
	\caption[A]{\small Illustration of subcase 3.2}\label{Fig.6}
\end{figure}

Subcase 3.2.2. Assume that \(\{s_0r_0, w_0t_0\}\) is as shown in Figure \ref{Fig.2}(b). 

Since \(p(w_0) \neq p(t_0)\), it follows that \(p(s_0) = p(w_0)\) or \(p(s_0) = p(t_0)\). W.l.o.g., let \(p(s_0) = p(t_0)\). Thus, we have \(p(s_0) = p(t_0) \neq p(w_0) = p(r_0)\), and \(p(a_1) = p(d_1) \neq p(c_1) = p(b_1)\). If \(\{a_1, c_1, b_1, d_1\} \cap V(M_1) = \emptyset\), then, using Lemma \ref{a7} when \(|M_1| = 1\) or Theorem \ref{a8} when \(M_1 = \emptyset \), there exists a spanning 3-path \( P^1_{a_1c_1} + P^1_{b_1d_1}+P^1_{yz_1} \) of \(Q^1_{n-1}\) passes through \(M_1\). The desired Hamiltonian path is $P_{xy}:= P_{xz_0}^0 + P^1_{a_1c_1} + P^1_{b_1d_1}+P^1_{yz_1} + \{s_0r_0, w_0t_0, a_0a_1, b_0b_1, c_0c_1, d_0d_1, z_0z_1\} - \{s_0a_0, r_0b_0, w_0c_0, t_0d_0\}$ in \(Q_n\).

Otherwise, \(|\{a_1, c_1, b_1, d_1\} \cap V(M_1)| = 1\), implying \(|M_1| = 1\).  Let \(M_1 = \{e_1\}\). Since \(p(a_1) = p(d_1) \neq p(c_1) = p(b_1)\), it follows that \(a_1b_1 = e_1\) or \(c_1d_1 = e_1\). If \(d_{P_{xz_0}^0}(t_0, r_0) = 1\), then because \(p(a_1) \neq p(c_1)\) and \(a_1c_1 \neq e_1\), using Lemma \ref{a5}, a spanning 2-path \(P^1_{a_1c_1} + P^1_{yz_1}\) exists in \(Q^1_{n-1}\) passes through $e_1$. The Hamiltonian path is \( P_{xy}:= P_{xz_0}^0  + P^1_{a_1c_1} + P^1_{yz_1}+ \{s_0r_0, w_0t_0, a_0a_1, c_0c_1, z_0z_1\} - \{s_0a_0, w_0c_0, r_0t_0\}\) in \(Q_n\). If \(d_{P_{xz_0}^0}(t_0, r_0) \neq 1\), let \(d_0^*\) be the neighbor of \(t_0\) on \(P_{xz_0}^0\) that is distinct from \(d_0\). Since \(p(t_0) \neq p(r_0)\) by assumption, we have \(p(d_0^*) \neq p(b_0)\), which implies \(d_0^* \neq b_0\). Note that \(a_0, b_0, c_0, d_0,\) and \(d_0^*\) are distinct. When \(a_1b_1 = e_1\), it follows that \(\{c_1, d^*_1, z_1, y\} \cap V(e_1) = \emptyset\). When \(c_1d_1 = e_1\), we have \(\{a_1, b_1, z_1, y\} \cap V(e_1) = \emptyset\). Since \(p(d^*_1) = p(t_0) \neq p(w_0) = p(c_1)\), by Lemma \ref{a7}, a spanning 3-path \( P^1_{a_1b_1} + P^1_{c_1d^*_1} +P^1_{yz_1} \) exists in \(Q^1_{n-1}\). The desired Hamiltonian path is \( P_{xy}:= P_{xz_0}^0 + P^1_{a_1b_1} + P^1_{c_1d^*_1} +P^1_{yz_1} + \{s_0r_0, w_0t_0, a_0a_1, b_0b_1, c_0c_1, d_0^*d^*_1, z_0z_1\} - \{s_0a_0, r_0b_0, w_0c_0, t_0d^*_0\}\) in \(Q_n\) containing $M$, as depicted in Figure \ref{Fig.6}(b).

Subcase 3.3. \( x, y \in V(Q_{n-1}^0) \).  

The proof is similar to Subcase 2.3.

Case 4: $|M_0| = 3n - 13$.

Subcase 4.1. $x, y\in V(Q_{n-1}^0)$.

Choose three edges $e, f, g\in M_0$, since $|M_0 \setminus \{e, f, g\}| = 3(n - 1) - 13$, and using induction a Hamiltonian path $P^0_{xy}$ exists in $Q_{n-1}^0$ containing $(M_0 \setminus \{e, f, g\})$. 

Let $s_0r_0, w_0t_0, u_0v_0$ be the chosen edges. If the set of edges $\{s_0r_0, w_0t_0, u_0v_0\}\cap E(C_0) \neq \emptyset$, then the result follows directly from Case 1, Case 2, or Case 3. If the set of edges $\{s_0r_0, w_0t_0, u_0v_0\} \cap E(P^0_{xy}) = \emptyset$, there are four distinct cases up to isomorphism, as shown in Figure \ref{Fig.7}. Let consider $a_0, b_0, c_0, d_0, m_0, n_0$ denote the neighbors of the vertices $s_0, r_0, w_0, t_0, u_0, v_0$ on the path $P^0_{xy}$, respectively. Since $M$ is a matching and $\{s_0r_0, w_0t_0, u_0v_0\} \subseteq M_0$, the vertices $a_0, b_0, c_0, d_0, m_0, n_0$ are distinct vertices, and the edges $\{s_0a_0, r_0b_0, w_0c_0, t_0d_0, u_0m_0, v_0n_0\}\cap M= \emptyset$. If $\{s_0r_0, w_0t_0, u_0v_0\}$, is as on Figure \ref{Fig.7}(a) or (c), let define $k_1, k_2, k_3, k_4, k_5, k_6$ correspond to the vertices $a_1, b_1, c_1, d_1, m_1, n_1$, respectively. For Figure \ref{Fig.7}(b), define  $k_1, k_2, k_3, k_4, k_5, k_6$ represent the vertices $a_1, d_1, b_1, m_1, c_1, n_1$. Similarly, in Figure \ref{Fig.7}(d), define $k_1, k_2, k_3, k_4, k_5, k_6$ to vertices $a_1, c_1, b_1, n_1, d_1, m_1$, respectively.
\begin{figure}[H]
	\centering
	\includegraphics[width=0.8\linewidth, height=0.43\textheight]{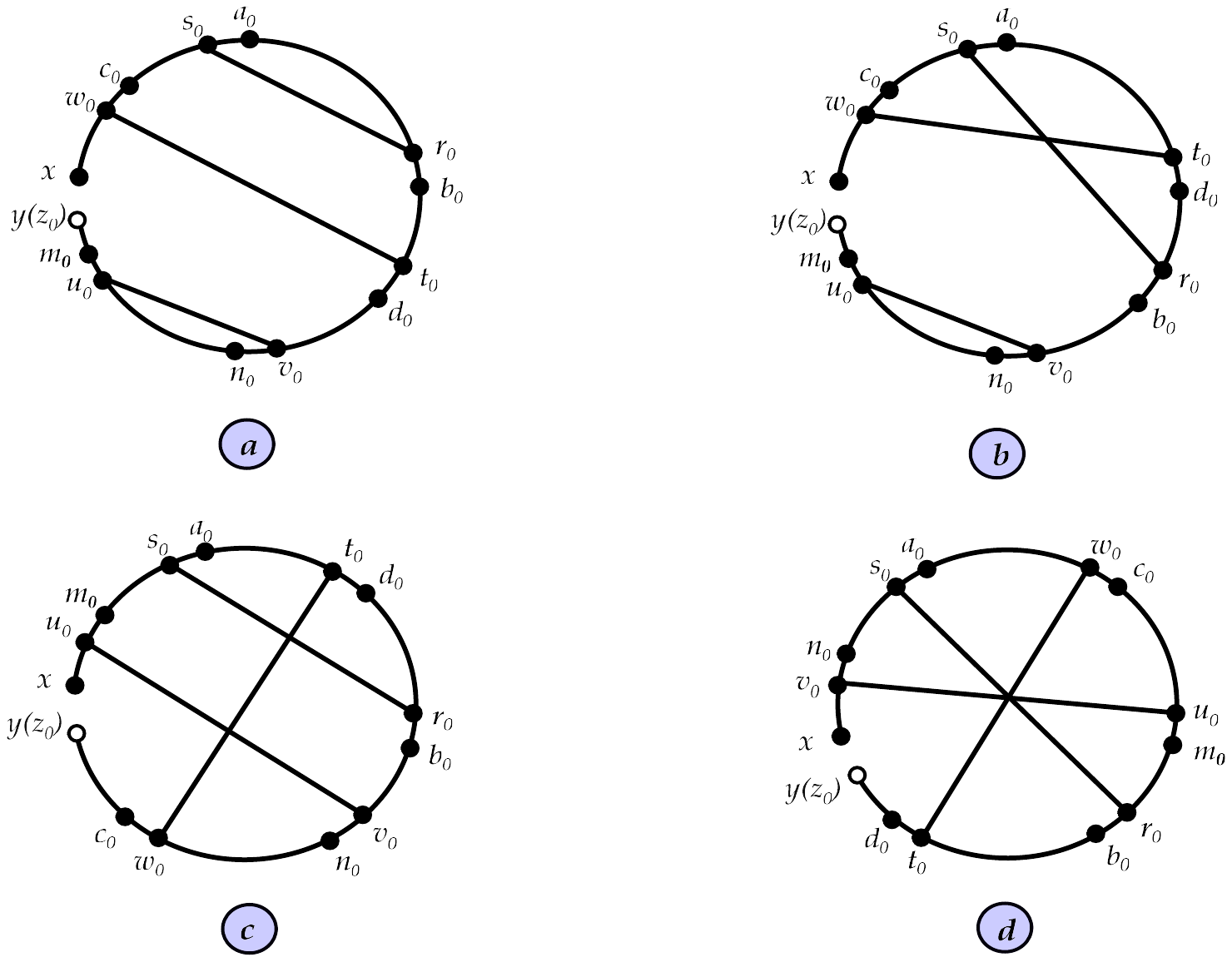}
	\caption[A]{\small Four possibilities of edges $\{s_0r_0, w_0t_0, u_0v_0\}$ and its neighbors.}\label{Fig.7}
\end{figure}

Since, $p(a_1) \neq p(b_1)$, $p(c_1) \neq p(d_1)$, and $p(m_1) \neq p(n_1)$, the set $\{\{k_1, k_2\}, \{k_3, k_4\}, \{k_5, k_6\}\}$ forms a balanced pair collection. Since $n - 1 > 9$, by Theorem \ref{a8}, a spanning 3-path $P^1_{k_1k_2} + P^1_{k_3k_4} + P^1_{k_5k_6}$ exists in $Q_{n-1}^1$. Therefore, the desired Hamiltonian path is $P_{xy}:= P^0_{xy} + P^1_{k_1k_2} + P^1_{k_3k_4} + P^1_{k_5k_6} + \{s_0r_0, w_0t_0, u_0v_0, a_0a_1, b_0b_1, c_0c_1, d_0d_1, m_0m_1, n_0n_1\} - \{s_0a_0, r_0b_0, w_0c_0, t_0d_0, u_0m_0, v_0n_0\}$ in $Q_n$ containing $M$.

Subcase 4.2. $x \in V(Q_{n-1}^0)$ and $y \in V(Q_{n-1}^1)$.

Since $|M_0| = 3n - 13$ three edges $e, f, g\in M_0$. Choose a vertex $z_0$ in $Q_{n-1}^0$ satisfying $p(z_0) \neq p(x)$ and  $xz_0\notin M_0$. Moreover, $p(z_1) \neq p(y)$. Since $|M_0 \setminus \{e, f, g\}| = 3(n - 1) - 13$, using induction, a Hamiltonian path $P^0_{xz_0}$ exists in $Q_{n-1}^0$ passing through $(M_0 \setminus \{e, f, g\})$. 

Let $s_0r_0, w_0t_0, u_0v_0$ be the chosen edges. If the set $\{s_0r_0, w_0t_0, u_0v_0\}\cap E(P^0_{xz_0}) \neq \emptyset$, then the result follows directly from Case 1, Case 2, or Case 3. If $\{s_0r_0, w_0t_0, u_0v_0\} \cap E(P^0_{xz_0}) = \emptyset$, there are four distinct cases up to isomorphism, as depicted in Figure \ref{Fig.7}. Let $a_0, b_0, c_0, d_0, m_0, n_0$ denote the neighbors of the vertices $s_0, r_0, w_0, t_0, u_0, v_0$ on the path $P^0_{xz_0}$, respectively. Since $M$ is a matching and $\{s_0r_0, w_0t_0, u_0v_0\} \subseteq M_0$, the vertices $a_0, b_0, c_0, d_0, m_0, n_0$ are all distinct, and $\{s_0a_0, r_0b_0, w_0c_0, t_0d_0, u_0m_0, v_0n_0\}\cap M= \emptyset$. If $\{s_0r_0, w_0t_0, u_0v_0\}$, is as on Figure \ref{Fig.7}(a) or (c), let $k_1, k_2, k_3, k_4, k_5, k_6$ correspond to the vertices $a_1, b_1, c_1, d_1, m_1, n_1$, respectively. In Figure \ref{Fig.7}(b), let us define $k_1, k_2, k_3, k_4, k_5, k_6$ denote the vertices $a_1, d_1, b_1, m_1, c_1, n_1$. Similarly, in Figure \ref{Fig.7}(d), denote $k_1, k_2, k_3, k_4, k_5, k_6$ the vertices $a_1, c_1, b_1, n_1, d_1, m_1$.

Since, $p(a_1) \neq p(b_1)$, $p(c_1) \neq p(d_1)$, $p(z_1)\neq p(y)$ and $p(m_1)\neq p(n_1)$, the set of vertices in $Q_{n-1}^1$ is $\{\{k_1, k_2\}, \{k_3, k_4\}, \{k_5, k_6\},\{z_1, y\}\}$ forms a balanced pair set. Since $n - 1 > 9$, by Theorem \ref{a8}, a spanning 4-path $P^1_{k_1k_2} + P^1_{k_3k_4} + P^1_{k_5k_6} + P^1_{z_1y}$ exists in $Q_{n-1}^1$. Therefore, the desired Hamiltonian path is
$P^0_{xz_0} + P^1_{k_1k_2} + P^1_{k_3k_4} + P^1_{k_5k_6} + +P^1_{yz_1} + \{s_0r_0, w_0t_0, u_0v_0, a_0a_1, b_0b_1, c_0c_1, d_0d_1, m_0m_1, n_0n_1, z_0z_1\} - \{s_0a_0, r_0b_0, w_0c_0, t_0d_0, u_0m_0, v_0n_0\}$ passing through the desired matching $M$ in the given hypercube $Q_n$.

Subcase 4.3. $x, y\in V(Q_{n-1}^1)$.

The proof is similar to Subcase 2.3.

This is the desired proof of the problem.
\subsection*{Data availability}
{\small No new data were generated or analyzed in support of this research.}
\subsection*{Conflict of interest}
{\small The authors have no conflict of interest.}
\subsection*{Acknowledgments} 
{\small We will express our appreciation to the anonymous referees for their productive feedback to improve the clarity of paper.}

\end{document}